\documentclass[reqno,11pt,a4paper]{amsart}

\usepackage[all]{xy}
\usepackage{amsthm}
\usepackage{amssymb}
\usepackage{amsmath,amscd}
\usepackage{amsfonts}
\usepackage{mathrsfs}
\usepackage{threeparttable}
\usepackage{tikz}
\usepackage{lscape}
\usepackage{mathtools}
\usepackage[foot]{amsaddr}

\usepackage[hidelinks]{hyperref}

\def\cal{\mathcal}

\def\AA{{\Bbb A}}

\def\NN{{\Bbb N}}
\def\PP{{\Bbb P}}

\def\ZZ{{\Bbb Z}}

\def\11{{1\kern-3.5pt 1}}
\def\mumu{{\mu\kern-4.2pt\mu}}

\def\boxtimes{\setbox0\hbox{$\Box$}\copy0\kern-\wd0\hbox{$\times$}}

\def\diag{\operatorname {diag}}

\def\Ext{\operatorname {Ext}}

\def\GL{\operatorname {GL}}

\def\id{\operatorname {id}}

\def\Ker{\operatorname {ker}}

\def\Spec{\operatorname {Spec}}

\def\Aut{\operatorname{Aut}}

\def\det{\operatorname{det}}
\def\dim{\operatorname{dim}}

\def\Ext{\operatorname{Ext}}

\def\uExt{\operatorname{\underline{Ext}}}

\def\GL{\operatorname{GL}}
\def\gldim{\operatorname{gldim}}
\def\grmod{\operatorname{grmod}}

\def\Ker{\operatorname{Ker}}

\def\Proj{\operatorname{Proj}}

\def\rank{\operatorname{rank}}

\def\tails{\operatorname{tails}}

\def\tors{\operatorname{tors}}

\def\uExt{\operatorname{\underline{Ext}}}

\def\G{\mathop{\underline{\underline{\it \Gamma}}}\nolimits}

\def\l{\leftarrow}

\let\oldtext\text
\def\text#1{\oldtext{\normalshape #1}}

\def\e{\epsilon}

\def\l{\lambda}
\def\s{\sigma}
\def\t{\tau}

\def\G{\Gamma}

\def\cA{{\cal A}}

\def\cD{{\cal D}}

\def\cF{{\cal F}}

\def\cP{{\cal P}}

\def\cV{{\cal V}}

\def\<{\langle}
\def\>{\rangle}

\def\uCM{\underline {\operatorname
{CM}}}

\def\sgn{\operatorname{sgn}}

\def\sA{\mathscr A}
\def\sB{\mathscr B}
\def\sC{\mathscr C}
\def\sD{\mathscr D}

\def\Projn{\operatorname{Proj_{nc}}}

\newtheorem{lemma}{Lemma}[section]
\newtheorem{proposition}[lemma]{Proposition}
\newtheorem{theorem}[lemma]{Theorem}
\newtheorem{corollary}[lemma]{Corollary}

{

}

\theoremstyle{definition}

\newtheorem{example}[lemma]{Example}
\newtheorem{definition}[lemma]{\sl Definition}

\theoremstyle{remark}

\newtheorem{remark}[lemma]{Remark}

\begin{document}

\pagenumbering{arabic}

\title
{Clifford Quadratic Complete Intersections}

\author{Haigang Hu$^{1}$ and Izuru Mori$^{2}$}

\address{${^1}$ School of Mathematical Sciences, University of Science and Technology of China, Hefei Anhui 230026, CHINA}

\email{huhaigang@ustc.edu.cn}

\email{huhaigang\_phy@163.com \textnormal{(H. Hu)}}

\address{${^2}$ Department of Mathematics, Faculty of Science, Shizuoka University, Shizuoka 422-8529, JAPAN}

\email{mori.izuru@shizuoka.ac.jp \textnormal{(I. Mori)}}

\keywords{Clifford algebra, quantum quadratic complete intersection, point variety, characteristic variety}

\thanks {{\it 2020 Mathematics Subject Classification}: 16E65, 16S38, 16W50}

\thanks {The second author was supported by Grant-in-Aid for Scientific Research (C) 20K03510.}

%\thanks {R04a R04nCQPA2}

%\noindent \today

\begin{abstract} 
In this paper, we define and study Clifford quadratic complete intersections.  After showing some properties of Clifford quantum polynomial algebras, we show that there is a natural one-to-one correspondence between Clifford quadratic complete intersections and commutative quadratic complete intersections.
As an application, we give a classification of Clifford quadratic complete intersections in three variables in terms of their characteristic varieties.    
\end{abstract}

\maketitle

%\tableofcontents

\section{Introduction} 

Throughout this paper, we work over an algebraically closed field $k$ of characteristic 0.   The most basic object in algebraic geometry is a projective space, which is a projective scheme associated to a polynomial algebra.  In noncommutative algebraic geometry, a quantum polynomial algebra (Definition \ref{def.qpa}) is a noncommutative analogue of a polynomial algebra, so a quantum projective space, which is defined as a noncommutative projective scheme associated to a quantum polynomial algebra, is a basic object of study.  In fact, the classification of quantum polynomial algebras is one of the major projects in noncommutative algebraic geometry (see \cite{ATV}.)  Quadric hypersurfaces may be the second most basic objects in algebraic geometry. In fact, the classification of quadric hypersurfaces is elementary by Sylvester (see Lemma \ref{lem.sy11}). 
Studying and classifying noncommutative quadric hypersurfaces are active projects in noncommutative algebraic geometry (see \cite{HY,H,HMM,SmV}.)   As a first step, noncommutative conics were completely classified in \cite{HMM} recently.  

Motivated by this, in this paper, we study noncommutative analogues of quadratic complete intersections. 
It is sensitive to define what should be a noncommutative analogue of a complete intersection.  Several options were proposed in \cite{CV1,KKZ,V1}.  We adopt a rather naive definition in this paper, which is not exactly the same as the options provided in the literature (see Remark \ref{rem.cseq}).
Recall that a (commutative) quadratic complete intersection (of length $r$ in $n$ variables) is defined as 
$$
B(F) = k[u_1,\dots,u_n]/(f_1,\dots,f_r), 
$$
where $f_1, \dots, f_r \in k[x_1,\dots,x_n]_2$ is a regular sequence which is uniquely determined by the sequence
$$
F = (F_1, \dots, F_r) \in M_n(k)^{\times r}
$$
of symmetric matrices. When $r=1$, $B(F)$ is a quadric hypersurface.  When $r\geq 2$, the classification of $B(F)$ is classical but highly non-trivial (see \cite {HP,W}). 
One approach to study $B(F)$ is to study its quadratic dual $B(F)^!$. The main purposes of this paper are to give a detailed proof to show that $B(F)^!$ is a noncommutative quadratic complete intersection embedded into a Clifford quantum projective space (Theorem \ref{thm.sq}), and to study $B(F)^!$ by using geometric methods (Section 4).

For a sequence
$$F = (F_1, \dots, F_n) \in M_n(k)^{\times n}$$
of symmetric matrices, 
we may define the  graded Clifford algebra associated to $F$ by
\begin{align*}
C(F): = k\<x_1, \dots, x_n, y_1, \dots, y_n\>/  (x_ix_j+x_jx_i-\sum_{m=1}^n(F_m)_{ij}y_m, & \\
x_iy_j-y_jx_i, y_iy_j-y_j & y_i)_{1\leq i, j\leq n}
\end{align*}
where $\deg x_i=1, \deg y_j=2$.   A graded Clifford algebra has been an active area of study (see \cite{AL,CYZ,LeB,SV}).  This notion was even extended to a graded skew Clifford algebra and intensively studied in \cite{CV1,CV2,V1,V2}.
In noncommutative algebraic geometry, Clifford quantum polynomial  algebras play the role of the ambient spaces where noncommutative quadratic complete intersections are embedded into, so after preparing some preliminary results in Section 2, 
we first study Clifford quantum polynomial algebras in Section 3.  The following result is folklore (see \cite{AL,CV1,LeB}), which motivates many researchers to study Clifford quantum polynomial algebras.

 \begin{theorem} \label{thm.1} 
 %%% [Corollary \ref{cor.cbbd}]
 For a sequence $F= (F_1, \dots, F_n)\in M_n(k)^{\times n}$ of linearly independent symmetric matrices,
 the following are equivalent:
\begin{enumerate}
\item{} $B(F)$ is a (commutative) quadratic complete intersection.
\item{} $C(F)$ is a 
% n $n$-dimensional 
Clifford quantum polynomial algebra.
% \item{} $C(F)$ is a domain. 
\end{enumerate} 
 \end{theorem}

% The following is also an interesting characterization of a Clifford quantum polynomial algebra. 

In Section 3.2, we show that every Clifford quantum polynomial algebra can be realized as a derivation quotient algebra of a symmetric superpotential (Proposition \ref{prop.symd}), 
% Using a notion of superpotential, we prove 
which leads the following result. 

 \begin{theorem}
 % \begin{enumerate}
 % \item{} {\rm (Corollary \ref{cor.clco})} A quantum polynomial algebra $S$ is a graded Clifford algebra if and only if the quadratic dual $S^!$ is commutative.
 % \item{} 
 {\rm (Corollary \ref{cor.symd})} An $n$-dimensional Clifford quantum polynomial algebra is Calabi-Yau if and only if $n$ is odd. 
 % \end{enumerate}
 \end{theorem}

We say that a graded algebra $A$ is a {\it Clifford (resp. Calabi-Yau) quadratic complete intersection of length $r$ in $n$ variables} if $A = S/(f_1, \dots, f_r)$ where $S$ is an $n$-dimensional Clifford (resp. Calabi-Yau) quantum polynomial algebra and $f_1, \dots, f_r \in Z(S)_2$ is a regular sequence. 
In particular, when $r=1$, we call $S/(f)$ a noncommutative quadric hypersurface, which is a main object of our project.  On the other hand, the case $r=n$ was intensively studied in \cite{CV1,V1}.
One of the motivations of this project is to study and classify noncommutative quadric hypersurfaces in Calabi-Yau quantum projective spaces. 
By the above theorem, Clifford quadratic complete intersections can be embedded into Calabi-Yau quantum projective spaces when $n$ is odd, so it is interesting to study and classify them.
The following theorem may be known by experts,  but, to our knowledge, it has not appeared in the literature except for the case $r=0, n$. Since it is useful in the classification, we will give a detailed proof in Section 3.4.
%We show the following result in Section 3.4, which is useful in the classification.

\begin{theorem}[Theorem \ref{thm.sq}]
There is a one-to-one correspondence 
$$
\begin{CD}
\left\{ \begin{tabular}{c} \text{Clifford quadratic }\\ \text{complete intersections } \\ \text{ of length $r$ } \\ \text{in $n$ variables} \end{tabular} \right\} / \cong 
& \;\; \xleftrightarrow{1:1} \;\; &
\left\{  \begin{tabular}{c} \text{commutative quadratic }\\  \text{complete intersections } \\ \text{of length $n-r$ } \\ \text{in $n$ variables} \end{tabular}  \right\}/ \cong
\end{CD}
$$
for all $n$ and $0 \leq r \leq n$ by taking quadratic dual.
\end{theorem}

The key to prove the above theorem is to 
% show that there are enough central regular sequences in 
compute the center $Z(S)$ explicitly for a Clifford quantum polynomial algebra $S$, which is done in Section 3.3.
In fact, it is rather surprising that every sequence of linearly independent elements $f_1, \dots, f_r\in Z(S)_2$ forms a regular sequence (Proposition \ref{prop.sq}).  The idea of the above theorem was used in classifying noncommutative conics in \cite{HMM}.  Although it is not easy to classify Clifford quantum polynomial algebras, the above theorem tells that it may be possible to classify Clifford quadratic complete intersections with relations.   In fact, there are infinitely many Clifford quantum projective planes up to isomorphism, on the other hand, there are only 6 noncommutative conics embedded into Clifford quantum projective planes up to isomorphism  (see Corollary \ref{cor.123})!

One of the basic ideas of noncommutative algebraic geometry is to use geometry to study noncommutative algebras.  
%{\tcr Artin, Tate and Van den Bergh developed a geometric method to study $3$-dimensional quantum polynomial algebras. Recently, Matsuno and the authors applied the geometric method to the classification of noncommutative conics. It turns out that this geometric method is also useful in the study of }
In Section 4, we provide a geometric method to study Clifford quadratic complete intersections.
By \cite{ATV}, every 3-dimensional quantum polynomial algebra $S$ determines and is determined by a geometric pair $(E, \s)$ where $E$ is a projective scheme and $\s\in \Aut E$.  We say that a quadratic algebra $A$ satisfies condition (G1) if it determines a geometric pair $\cP(A)=(E, \s)$ (when $E$ is a projective variety).  In this case, we call $E$ the {\it point variety} of $A$ (Definition \ref{def.geoalg}).

Every noncommutative conic $A$ in a Calabi-Yau quantum projective plane satisfies the condition (G1) (\cite [Proposition 4.5]{HMM}), and the point variety $E_A$ of $A$ 
was useful to classify noncommutative conics (\cite[Theorem 4.14]{HMM}). We generalize the calculation method of $E_A$ 
for Clifford quadratic complete intersections $A$:

\begin{theorem} Let $A = S/(f_1, \dots, f_r)$ be a Clifford quadratic complete intersection where $f_1, \dots, f_r \in Z(S)_2$ is a central regular sequence.
\begin{enumerate}
\item{}  {\rm(Theorem \ref{thm.0G1})} $S$ satisfies the condition (G1).
\item{} {\rm(Theorem \ref{thm.cen})} There are $g_1, \dots, g_r \in S_1$ such that $g_i^2 = f_i$ for $1\leq i\leq r$. (Every central element in $Z(S)_2$ is reducible!)
\item{} {\rm(Theorem \ref{thm.SsA})} $A$ satisfies the condition (G1).  Moreover, if $\mathcal{P}(S) = (E,\sigma)$,  then $\mathcal{P}(A) = (E_A, \sigma_A)$ is given by 
$$
E_A=(E\cap \cV(g_1, \dots, g_r))\cup \s(E\cap \cV(g_1, \dots, g_r)), \; \s_A=\s|_{E_A}.
$$
\end{enumerate}
\end{theorem}

For a Clifford quadratic complete intersection $A=C(F)/(g_1^2, \dots, g_r^2)$ where $F$ is normalized (Definition \ref{def.norm}) and $g_1, \dots , g_r\in C(F)_1$, there is an interesting result connecting the point variety $E_A$ and the {\it characteristic variety} of $A$ defined as follows: 
Define 
$$
X^{(s)}(F)  : =\{(\l_1, \dots, \l_n)\in \PP^{n-1}\mid \rank (\l_1F_1+\cdots +\l_nF_n)< s\}.
$$
We call 
$$
X_{A}^{(s)}: = X^{(s)}(F)  \cap \mathcal{V}(\tilde g_1, \dots, \tilde g_r)
$$
the characteristic varieties of $A$ (see Definition \ref{def.chva}).
Then there is a double cover map of varieties
$$
\Phi: E \to X^{(3)}(F)
$$
which restricts to a double cover map of varieties
\begin{equation*} 
\Phi|_{E_{A}}: E_{A} \to X_{A}^{(3)}
\end{equation*}
(Theorem \ref{thm.0G1}, Theorem \ref{thm.XA}).

As an application of the results mentioned above,
we give complete classifications of the point varieties $E_A$ and the characteristic varieties $X_{A}^{(3)}$ for Clifford quadratic complete intersections of length  $ r =1,2,3$ in $3$ variables (Table \ref{tab.pa}). \\

\noindent {\bf Acknowledgement.} We would like to thank P. Belmans and M. Vancliff for the helpful information. 

\section{Preliminaries} 

\subsection{Definitions and notations}

All algebras and vector spaces are over $k$ unless otherwise mentioned. A connected graded algebra $A$ is a positively graded algebra $A = \bigoplus_{i \geq 0} A_i$ such that $A_0 = k$. 
Let $A$ be a right noetherian connected graded algebra.  
We denote by   $\grmod A$ the category of finitely generated graded right $A$-modules.  A morphism in $\grmod A$ is a right $A$-module homomorphism preserving degrees. For $M \in \grmod A$ and $j\in \ZZ$, we define the shift $M(j) \in \grmod A$ by $M(j)_i : = M_{j+i}$ for $i \in \mathbb{Z}$.  We define the Hilbert series of $M$ by  $H_M(t) : = \sum_{i \in \mathbb{Z}} (\dim_{k} M_i) t^i \in \mathbb{Z}[[t,t^{-1}]]$.  For $M, N \in \grmod A$, we write $\Ext^i_A(M, N):=\Ext^{i}_{\grmod A}(M,N)$
and 
$$
\uExt^i_A (M,N): = \bigoplus_{j \in \mathbb{Z}} \Ext^i_A (M, N(j)).
$$

A quantum polynomial algebra defined below is a noncommutative analogue of a commutative polynomial algebra in noncommutative algebraic geometry.

\begin{definition} \label{def.qpa}
A noetherian connected graded algebra $S$ is called an $n$-dimensional quantum polynomial algebra if 
\begin{enumerate}
\item{} $\gldim S=n$, 
\item{} $\uExt^i_S(k, S)=\begin{cases}  k(n) & \textnormal {if } i=n, \\
0 & \textnormal{otherwise}, \end{cases}$ and
\item{} $H_S(t)=(1-t)^{-n}$.
\end{enumerate} 
\end{definition}

\begin{remark}{\rm
By {\cite[Theorem 5.11]{Sm}}, every quantum polynomial algebra is a {\it Koszul} algebra, in particular, a quadratic algebra (for the definition of Koszul, see Section \ref{subsec.qud}). 
}
\end{remark}

Let $S$ be an $n$-dimensional quantum polynomial algebra. The canonical module of $S$ is defined by 
$$
\omega_S : = \lim_{i \to \infty} \uExt^n_S(S/S_{\geq i}, S),
$$ 
which has a natural graded $S$-$S$-bimodule structure.

\begin{definition} \label{def.cya}
An $n$-dimensional quantum polynomial algebra $S$ is called an {\it $n$-dimensional Calabi-Yau quantum polynomial algebra} if $\omega_S\cong S(-n)$ as graded $S$-$S$ bimodules.
\end{definition}

We restate a definition of  a graded Clifford algebra. 

\begin{definition} \label{def.gcliffod}
For a sequence  $F=(F_1, \dots,  F_n)$ of linearly independent symmetric matrices of size $n$,  we define  the graded algebra by 
\begin{align*}
C(F): = k\<x_1, \dots, x_n, y_1, \dots, y_n\>/  (x_ix_j+x_jx_i-\sum_{m=1}^n(F_m)_{ij}y_m, & \\
x_iy_j-y_jx_i, y_iy_j-y_j & y_i)_{1\leq i, j\leq n}
\end{align*}
where $\deg x_i=1, \deg y_j=2$.   
We call $C(F)$ a {\it graded Clifford algebra}.  If it is also an $n$-dimensional quantum polynomial algebra, then we call $C(F)$ an {\it $n$-dimensional Clifford quantum polynomial algebra.}  
\end{definition}

\begin{remark}
The range of subscripts is important and sensitive in this paper. 
Since $F_m$ are symmetric for $1 \leq m \leq n$, 
we may write
\begin{align*}
C(F): = k\<x_1, \dots, x_n, y_1, \dots, y_n\>/  (x_ix_j+x_jx_i-\sum_{m=1}^n(F_m)_{ij}y_m,& \\
x_iy_j-y_jx_i, y_iy_j-y_j  & y_i)_{1\leq i\leq j\leq n}.
\end{align*}
\end{remark} 

By \cite[Lemma 8.2]{ATV}, every graded Clifford algebra $C(F)$ is noetherian (see \cite[Proposition 2.3 (a)]{V1}).  
% As mentioned in the proof of \cite[Theorem 5]{S} (see also Lemma \ref{lem.c1}), $y_1, \dots, y_n$ are linearly independent. Also, 
% By definition, $y_1, \dots, y_n$ are central in $C(F)$.  
For every $u,v \in C(F)_1$, $uv+vu\in C(F)_2$ is a linear combination of central elements $y_1, \dots, y_n$, so $uv+vu$ is a central element in $C(F)$ (though possibly equals to $0$). 
% , and hence central.  A natural question is that whether  every central element $z \in C(F)_2$ is a linear combination of $y_m$? The answer is yes, we will give a detailed proof 
We will explicitly compute the center $Z(C(F))$ in Section \ref{center}.  

We propose the following definition for a  noncommutative analogue of the quadratic complete intersection.

\begin{definition} \label{def.qqci}
We say that $A$ is a {\it  noncommutative 
 (Calabi-Yau, Clifford) quadratic complete intersection of length $r$ in $n$ variables} if $A=S/(f_1, \dots, f_r)$ for some $n$-dimensional (Calabi-Yau, Clifford) quantum polynomial algebra $S$
and some central regular sequence $f_1, \dots, f_r\in Z(S)_2$.   We denote by 
\begin{enumerate}
\item[(1)] $\sA_{n, r}$ the set of isomorphism classes of  Calabi-Yau quadratic complete intersections of length $r$ in $n$ variables. 
\item[(2)] $\sB_{n, r}$ the set of isomorphism classes of commutative quadratic complete intersections of length $r$ in $n$ variables.
\item[(3)] $\sC_{n, r}$ the set of isomorphism classes of Clifford quadratic complete intersections of length $r$ in $n$ variables.
\end{enumerate}
\end{definition}

By definition, $\sA_{n, 0}$, (resp. $\sB_{n, 0}$ or $\sC_{n, 0}$) is the set of isomorphism classes of $n$-dimensional Calabi-Yau (resp. commutative or Clifford) quantum polynomial algebras. 

\begin{definition} \label{def.geoalg}
We say that a quadratic algebra $A = T(V)/(R)$ satisfies the condition (G1) if there is a pair $(E,\sigma)$ where $E \subset \mathbb{P}(V^*)$ is a projective variety and $\sigma\in \Aut E$ is a $k$-automorphism of $E$ such that
$$
{\mathcal{V}}(R) = \{ (p ,\sigma(p)) \in \mathbb{P}(V^*) \times \mathbb{P}(V^*) \mid p \in E \}.
$$ 
If $A$ satisfies $(G1)$,  then we write $\mathcal{P}(A) = (E, \sigma)$.  We call $(E,\sigma)$ the {\it geometric pair} of $A$, and $E$ the {\it point variety} of $A$. 
\end{definition}

By Artin, Tate and Van den Bergh \cite{ATV}, every $3$-dimensional quantum polynomial algebra satisfies the condition (G1).

\begin{example}\label{exm.3-dimgcliff}
By {\cite[Corollary 4.8]{SV}}, $S\in \sC_{3, 0}$ is a $3$-dimensional Clifford quantum polynomial algebra if and only if  $S$ is isomorphic to 
$$
S^{(a,b,c)}:=k\langle x,y,z \rangle/(yz + zy + ax^2, zx + xz + b y^2, xy + yx + cz^2)
$$
where 
$$
(a,b,c) \in \{(0,0,0), (1,0,0), (1,1,0), (\lambda, \lambda, \lambda) \mid \lambda \in k, \lambda^3 \neq 0,1,-8 \}.
$$
For 
each allowed choice of $(a,b,c)$ as above, the geometric pair $\mathcal{P}(S) = (E, \sigma)$ of $S$ is given in Table \ref{3dimcliqpa} of Section 5 (\cite[Table 1]{HMM}).

Algebras $$
S^{(\l,\l,\l)}=k\langle x,y,z \rangle/(yz + zy + \l x^2, zx + xz + \l y^2, xy + yx + \l z^2)
$$
of Type EC listed in Table \ref{3dimcliqpa} of Section 5 are special cases of 3-dimensional Sklyanin algebras.  Since there are infinitely many isomorphism classes of elliptic curves, we see that $\#(\sC_{3, 0})=\infty$ (see Theorem \ref{thm.gepa}).
\end{example}

\subsection{Clifford algebras} 

Recall the definition of a (classical) Clifford algebra over a commutative ring.

\begin{definition} 
For a commutative ring $R$ and a symmetric matrix $\cF\in M_n(R)$, we define an algebra by 
$$C_R(\cF):=R\<x_1, \dots, x_n\>/(x_ix_j+x_jx_i-\cF_{ij})_{1\leq i, j\leq n}.$$
We call $C_R(\cF)$ a {\it Clifford algebra} over $R$.  
\end{definition} 

We list some results on Clifford algebras. 

\begin{lemma}[{\cite[Theorem 7.1.6]{HO}}] \label{lem.HO}
$C_R(\cF)$ is free as an $R$-module with a basis
$$
\{ 1, x_{i_1} \cdots x_{i_t} \mid 1 \leq t \leq n, i_1 < \dots < i_t\}.  
$$ 
In particular, $\rank_R C_R(\cF) = 2^n$. 
\end{lemma}

\begin{lemma} \label{lem.FF'} 
Let $R$ be a commutative ring and $\cF, \cF'\in M_n(R)$ symmetric matrices.  If $\cF'=P^t\cF P$ for some $P\in \GL_n(R)$, then $C_R(\cF)\cong C_R(\cF')$. 
\end{lemma} 

\begin{proof} Let $P=(p_{ij})\in \GL_n(R)$ and define an $R$-algebra automorphism $\phi_P:R\<x_1, \dots, x_n\>\to R\<x_1, \dots, x_n\>$ by $\phi_P(x_j)=\sum_{i=1}^np_{ij}x_i$.  
Since 
\begin{align*}
& \phi_P(x_ix_j+x_jx_i-\cF'_{ij}) \\
= & \left(\sum_{s=1}^np_{si}x_s\right)\left(\sum_{t=1}^np_{tj}x_t\right)+\left(\sum_{t=1}^np_{tj}x_t\right)\left(\sum_{s=1}^np_{si}x_s\right)-(P^t\cF P)_{ij} \\
= & \sum_{1\leq s, t\leq n}p_{si}p_{tj}(x_sx_t+x_tx_s-\cF_{st}),
\end{align*}
$\phi_P$ induces a surjective ring homomorphism $C_R(\cF')\to C_R(\cF)$.  Applying the same argument to $\phi_P^{-1}=\phi_{P^{-1}}$, we have $C_R(\cF)\cong C_R(\cF')$.     
\end{proof}

\begin{lemma}[Sylvester's theorem {(\cite[Theorem IV.7]{N})}] \label{lem.FF'2} \label {lem.sy11}  
Let $K$ be a field such that $\operatorname{char} K\neq 2$. 
 For every symmetric matrix $A\in M_n(K)$ of rank $r$, there exists a diagonal matrix $A'=\diag (\l_1, \dots, \l_r, 0, \dots 0)\in M_n(K)$ where $\l_i\neq 0$ for every $i=1, \dots, r$ such that 
 $A' = P^t A P$ for some $P \in \GL_n(K)$.  
Moreover, if $K$ is algebraically closed, then we may even take $\l_i=1$ for every $i=1, \dots, r$.
\end{lemma} 

\begin{lemma}[{\cite[Theorem V.2.2]{L}}] \label{lem.cKf} 
 Let $K$ be a field such that ${\rm char } K\neq 2$, and $\cF\in M_n(K)$ a diagonal matrix.  If $\det \cF\neq 0$, then  
 $$Z(C_K(\cF))=\begin{cases} 
 K & \textnormal { if $n$ is even,} \\
 K+Kx_1\cdots x_n & \textnormal { if $n$ is odd.}\end{cases} $$
\end{lemma}

\begin{lemma} \label{lem.bac} 
Let $\phi:R\to S$ be a homomorphism of commutative rings $R, S$, and $\bar\phi:R\<x_1, \dots, x_n\>\to S\<x_1, \dots, x_n\>$ a natural extension of $\phi$. 
\begin{enumerate}
\item{} $\varphi:R\<x_1, \dots, x_n\>\otimes _RS\to S\<x_1,\dots, x_n\>; \: f\otimes s\mapsto \bar\phi(f)s$ is an isomorphism of rings. 
\item{}
For $f_1, \dots, f_r\in R\<x_1, \dots,x_n\>$, $\varphi$ induces an isomorphism
$$R\<x_1, \dots, x_n\>/(f_1, \dots, f_r)\otimes _RS\to S\<x_1,  \dots, x_n\>/(\bar\phi(f_1), \dots, \bar\phi(f_r))$$
of rings. 
\end{enumerate}
\end{lemma} 

\begin{proof} 
Left to the reader.
\end{proof}

Let $R$ be a commutative graded algebra and $\cF\in M_n(R)$ a symmetric matrix.  If $\cF_{ij}\in R_2$ for every $1\leq i, j\leq n$, then we may view $C_R(\cF)$ as a graded algebra by $\deg x_i=1$ for every $1\leq i\leq n$.  
The following result explaining why the algebra $C(F)$ in Definition \ref{def.gcliffod} is called a graded Clifford algebra is standard (see \cite[Section 4]{LeB}).  We will include a proof for the convenience of the reader since we will use it several times in this paper.) 

\begin{lemma} \label{lem.c1} 
If $F=(F_1, \dots, F_n)\in M_n(k)^{\times n}$ is a sequence of symmetric matrices, and $\cF:=F_1y_1+\cdots +F_ny_n\in M_n(R)$ where $R:=k[y_1, \dots, y_n]$, then there is an isomorphism $C(F)\to C_R(\cF)$ of graded algebras, which is also an isomorphism of $R$-algebras. 
\end{lemma}  

\begin{proof} 
By Lemma \ref{lem.bac} (1), 
\begin{align*}
R\<x_1, \dots, x_n\> & \cong k\<x_1, \dots, x_n\>\otimes _kR \\
& =k\<x_1, \dots, x_n\>\otimes _kk\<y_1, \dots, y_n\>/(y_iy_j-y_jy_i)_{1\leq i, j\leq n} \\
& \cong k\<x_1, \dots, x_n, y_1, \dots, y_n)/(x_iy_j-y_jx_i, y_iy_j-y_jy_i)_{1\leq i, j\leq n}.  
\end{align*}
Since $\cF_{ij}=\sum_{m=1}^n(F_m)_{ij}y_m$,  
\begin{align*}
C_R(\cF) & :=R\<x_1, \dots, x_n\>/(x_ix_j+x_jx_i-\cF_{ij})_{1\leq i, j\leq n} \\
& \cong k\<x_1, \dots, x_n, y_1, \dots, y_n\>/(x_ix_j+x_jx_i-\sum_{m=1}^n(F_m)_{ij}y_m, \\
& \ \ \ \ \ \ \ \ \ \ \ \ \ \ \ \ \ \ \ \ \ \ \ \ \ \ \ \ \ \ \ \ \ \ \ \ \ \ \ \ \ \ \ \ \ \  \ x_iy_j-y_jx_i, y_iy_j-y_jy_i)_{1\leq i, j\leq n} \\
& =:C(F). 
\end{align*}
Since the isomorphism of graded algebras constructed above fixes $y_j$ for every $1\leq j\leq n$, it is also an isomorphism of $R$-algebras. 
\end{proof}

\subsection{Normalization} 

For a sequence  $F=(F_1, \dots,  F_r)\in M_n(k)^{\times r}$ of linearly independent symmetric matrices,  we define 
a commutative graded algebras 
$B(F):=k[u_1, \dots, u_n]/(f_1, \dots, f_r)$ where $\deg u_i=1$ for $1\leq i\leq n$ and $f_m:=\sum _{1\leq i, j\leq n}(F_m)_{ij}u_iu_j\in k[u_1, \dots, u_n]_2$ for $1\leq m\leq r$.
  
\begin{definition} We say that a sequence $F=(F_1, \dots,  F_n)\in M_n(k)^{\times n}$  of symmetric matrices is base point free if $\cap _{m=1}^n\cV(f_m)=\emptyset$ in $\PP^{n-1}$ where $f_m=\sum_{1\leq i, j\leq n }(F_m)_{ij}u_iu_j\in k[u_1, \dots, u_n]_2$. 
\end{definition}

\begin{remark} \label{rem.bpf} 
A sequence $F=(F_1, \dots,  F_n)\in M_n(k)^{\times n}$ of symmetric matrices is base point free if and only if $B(F)\in \sB_{n, n}$.  This fact was even extended to a skew polynomial algebra in \cite[Corollary 11 (Corrigendum)]{CV1}.
% \cite[Corollary 3.6]{CV1}???) } 
% for a more general version of this result.) }
\end{remark}

\begin{definition} \label{def.clP} 
Let $F=(F_1, \dots,  F_r), F'=(F'_1, \dots,  F'_r)\in M_n(k)^{\times r}$ be sequences of linearly independent symmetric matrices.  
\begin{enumerate}
\item{} We write $F\sim_sF'$ if there exists $P=(p_{ij})\in \GL_r(k)$ such that $F'_j=\sum _{i=1}^rp_{ij}F_i$ for every $j=1, \dots, r$. 
\item{} We write $F\sim_tF'$ if there exists $P=(p_{ij})\in \GL_n(k)$ such that $F'_j=P^tF_jP$
for every $j=1, \dots, r$. 
\item{} We write $F\sim_{st}F'$ if there exists $F''$ such that $F\sim _sF''\sim _tF'$.   
\end{enumerate}
\end{definition}

\begin{lemma} \label{lem.bP} Let $F=(F_1, \dots,  F_r), F'=(F'_1, \dots,  F'_r)\in M_n(k)^{\times r}$ be sequences of linearly independent symmetric matrices.  
\begin{enumerate}
\item{} If $F\sim _sF'$,  
 then $B(F)= B(F')$.
 \item{} If $F\sim _tF'$,  
 then $B(F)\cong B(F')$.
 \item{} $F\sim_{st}F'$ if and only if $B(F)\cong B(F')$.  
 \end{enumerate} 
\end{lemma} 

\begin{lemma} \label{lem.clP} Let $F=(F_1, \dots,  F_n), F'=(F'_1, \dots,  F'_n)\in M_n(k)^{\times n}$ be sequences of linearly independent symmetric matrices.  If $F\sim_{st}F'$, then $C(F)\cong C(F')$.  
\end{lemma} 
 
 \begin{proof}
If  $F\sim _sF'$ so that $F'_j=\sum _{i=1}^np_{ij}F_i$ for some $P=(p_{ij})\in \GL_n(k)$ for every $j=1, \dots, n$, then 
$$\sum_{j=1}^nF'_jy_j=\sum_{j=1}^n\left(\sum _{i=1}^np_{ij}F_i\right)y_j=\sum_{i=1}^nF_i\left(\sum _{j=1}^np_{ij}y_j\right),$$
so the map $\phi:C(F)\to C(F')$ defined by $\phi(x_i)=x_i, \phi(y_i)=\sum_{j=1}^np_{ij}y_j$ is an isomorphism of graded algebras. 

On the other hand, let $R=k[y_1, \dots, y_n]$, and $\cF=\sum _{m=1}^nF_my_m, \cF'=\sum _{m=1}^nF'_my_m\in M_n(R)$.  If $F\sim _tF'$ so that $F'_j=P^tF_jP$ for some $P\in \GL_n(k)$ 
for every $j=1, \dots, n$, then $\cF'=P^t\cF P$ for $P\in \GL_n(R)$.  Since $p_{ij}\in k=R_0$ for every $1\leq i, j\leq n$, $C_R(\cF)\cong C_R(\cF')$ as graded algebras by the proof of Lemma \ref{lem.FF'}, so 
$C(F)\cong C_R(\cF)\cong C_R(\cF')\cong C(F')$ as graded algebras by Lemma \ref{lem.c1}. 
 \end{proof} 

We prepare some elementary results on matrices.  

\begin{lemma} \label{lem.ac} 
Let $A=(a_{ij})\in M_m(k)$ be a symmetric matrix and $C=(c_{ij})\in M_{m, l}(k)$.  If $R^tAR+R^tC+C^tR=O$ for every $R=(r_{ij})\in M_{m, l}(k)$, then $A=O$ and $C=O$. 
\end{lemma} 

\begin{proof} Viewing as a polynomial in $k[r_{is}]_{1\leq i\leq m, 1\leq s\leq l}$, 
$$(R^tAR+R^tC+C^tR)_{ss}=\sum _{1\leq i, j\leq m}a_{ij}r_{is}r_{js}+2\sum _{i=1}^m
c_{is}r_{is}=0$$
for every $s=1, \dots, l$. 
Since $A$ is symmetric, $a_{ij}=c_{ij}=0$ for every $i, j$.
\end{proof} 

\begin{definition} \label{def.norm} We say that a sequence $F=(F_1, \dots,  F_n)\in M_n(k)^{\times n}$ of linearly independent symmetric matrices is {\it normalized} if $((F_i)_{jj})=2E_n$ where $E_n \in M_n(k)$ is the identity matrix. 
\end{definition} 

\begin{lemma} \label{lem.clP2} 
For every sequence  $F=(F_1, \dots,  F_n)\in M_n(k)^{\times n}$ of linearly independent symmetric matrices, there exists a normalized $F'=(F'_1, \dots,  F'_n)\in M_n(k)^{\times n}$ such that $F\sim_{st}F'$.
\end{lemma} 

\begin{proof} We first construct $F'\sim _{st}F $, $F'_m=\begin{pmatrix} A_m & * \\ * & * \end{pmatrix}$ where $A_m\in M_m(k)$ such that $(A_m)_{ii}=0$ for $i=1, \dots, m-1$ and $(A_m)_{mm}=1$ 
inductively so that 
$$((F'_i)_{jj})=\begin{pmatrix} 1 & * & \cdots & * \\
0 & 1 & \cdots & * \\
\vdots & \vdots & \ddots & * \\
0 & 0 & \cdots & 1 \end{pmatrix}.$$
Since $F_1\neq O$, we can construct $F'_1$ such that $(F'_1)_{11}=1$ by Lemma \ref{lem.sy11}.  Suppose that we can construct $F'_1, \dots, F'_m$ as above.  We may assume that 
$F'_{m+1}=\begin{pmatrix} A'_m & C'_m \\ {C'_m}^t & B'_m \end{pmatrix}$ where $A'_m\in M_m(k)$ such that $(A'_m)_{ii}=0$ for $i=1, \dots, m$ by using $\sim_s$.  If $B'_m\neq O$, then there exists $P=\begin{pmatrix} E_m & O \\ O & Q \end{pmatrix}\in \GL_n(k)$ such that   
\begin{align*}
& \begin{pmatrix} E_m & O \\ O & Q^t \end{pmatrix} \begin{pmatrix} A_l & * \\ * & * \end{pmatrix}\begin{pmatrix} E_m & O \\ O & Q \end{pmatrix} =\begin{pmatrix} A_l & * \\ * & * \end{pmatrix} \textnormal { for $l=1, \dots, m$, and}  \\
& \begin{pmatrix} E_m & O \\ O & Q^t \end{pmatrix} \begin{pmatrix} A'_m & C'_m \\ {C'_m}^t & B'_m \end{pmatrix}\begin{pmatrix} E_m & O \\ O & Q \end{pmatrix} =\begin{pmatrix} A'_m & C'_mQ \\ Q^t{C'_m}^t & Q^tB'_mQ \end{pmatrix}
\end{align*}
where $(Q^tB'_mQ)_{11}=1$ by Lemma \ref{lem.sy11}, so we are done.  If $B'_m=O$, then $A'_m\neq O$ or $C'_m\neq O$, so there exists $P=\begin{pmatrix} E_m & R \\ O & E_{n-m} \end{pmatrix}\in \GL_n(k)$ such that   
\begin{align*}
& \begin{pmatrix} E_m & O \\ R^t & E_{n-m} \end{pmatrix} \begin{pmatrix} A_l & * \\ * & * \end{pmatrix}\begin{pmatrix} E_m & R \\ O & E_{n-m} \end{pmatrix} =\begin{pmatrix} A_l & * \\ * & * \end{pmatrix} \textnormal { for $l=1, \dots, m$, and} \\
& \begin{pmatrix} E_m & O \\ R^t & E_{n-m} \end{pmatrix} \begin{pmatrix} A'_m & C'_m \\ {C'_m}^t & O \end{pmatrix}\begin{pmatrix} E_m & R \\ O & E_{n-m} \end{pmatrix} \\
=&\begin{pmatrix} A'_m & A'_mR+C \\ R^tA'_m+{C'_m}^t & R^tA'_mR+R^tC'_m+{C'_m}^tR \end{pmatrix}
\end{align*}
where $R^tA'_mR+R^tC'_m+{C'_m}^tR\neq O$ by Lemma \ref{lem.ac}, so we are done. 
It is easy to see that the above $F'$ can be normalized by using $\sim_s$.    
\end{proof}  

The following result is very useful in this paper. 

\begin{corollary} \label{cor.nor} 
For every $B\in \sB_{n, n}$, there exists a normalized $F$ such that $B\cong B(F)$.  
For every $S\in \sC_{n, 0}$, there exists a normalized $F$ such that $S\cong C(F)$.  
\end{corollary} 

\begin{proof} This immediately follows from Lemma \ref{lem.bP}, Lemma \ref{lem.clP}, and Lemma \ref{lem.clP2}.  
\end{proof}

\subsection{Regular sequences} 

\begin{definition} \label{q.abH} 
Let $a(t)=\sum _{i\in \NN}a_it^i, b(t)=\sum _{i\in \NN}b_it^i\in \ZZ[[t]]$.   
\begin{enumerate}
\item{} We write $a(t)\leq b(t)$ if $a_i\leq b_i$ for every $i\in \NN$.  
\item{} We write $a(t)< b(t)$ if $a(t)\leq b(t)$ and $a(t)\neq b(t)$ (that is,  $a_i\leq b_i$ for every $i\in \NN$ and $a_i< b_i$ for some $i\in \NN$).  
\item{} We write $\ZZ[[t]]^+:=\{a(t)\in \ZZ[[t]]\mid a(t)>0\}$.  
\end{enumerate}
\end{definition}

\begin{lemma} \label{lem.abc} 
Let $a(t), b(t), a_1(t), \dots, a_r(t)\in \ZZ[[t]]$.   
\begin{enumerate}
\item{} $(\ZZ[[t]], \leq )$ is a partially ordered set. 
\item{} 
If $a_1(t)\geq a_2(t)\geq \cdots \geq a_r(t)$ and $a_1(t)=a_r(t)$, then 
 $a_1(t)= a_2(t)=\cdots = a_r(t)$.
\item{} If $a(t)\leq b(t)$, then $a(t)c(t)\leq b(t)c(t)$ for every $c(t)\in \ZZ[[t]]^+$. 
\item{} If $a(t)< b(t)$, then $a(t)c(t)< b(t)c(t)$ for every $c(t)\in \ZZ[[t]]^+$. 
\end{enumerate}
\end{lemma}

\begin{remark} \label{rem.dKH}  Let $a(t), b(t)\in \ZZ[[t]]$.   
\begin{enumerate} 
\item{} Since $\ZZ[[t]]$ is an integral domain, for $c(t)\in \ZZ[[t]]$ such that $c(t)\neq 0$, $a(t)=b(t)$ if and only if $a(t)c(t)=b(t)c(t)$.  
\item{} If $a(t), b(t)\in \ZZ[[t]]^+$, then $a(t)b(t)\in \ZZ[[t]]^+$.  
\item{} 
Even if $a(t)^{-1}\in \ZZ[[t]]$ exists, it is not always the case that $a(t)\in \ZZ[[t]]^+$ implies $a(t)^{-1}\in \ZZ[[t]]^+$.  For example, for $d\in \NN^+$, we write $\dfrac{1}{1-t^d}:=\sum_{i\in \NN}t^{di}\in \ZZ[[t]]^+$.  Note that $\dfrac{1}{1-t^d}\cdot (1-t^d)=1$ in $\ZZ[[t]]$ but $1-t^d\not\in \ZZ[[t]]^+$.  
\end{enumerate}
\end{remark} 

Let $K$ be a locally finite $\NN$-graded vector space.  Note that $K=0$ if and only if $H_K(t)=0$, and $K\neq 0$ if and only if $H_K(t)>0$.
Applying to $K:=\Ker (A(-d) \to A; a\mapsto fa)$,  the following result is easy to see. (One direction follows from \cite[Lemma 2.1]{CV1}).

\begin{lemma} \label{lem.nrHn}
Let $A$ be a locally finite $\NN$-graded algebra and $f\in A_d$ a homogeneous normal element.  
\begin{enumerate}
\item{} $f$ is regular if and only if $H_{A/(f)}(t)=(1-t^d)H_A(t)$. 
\item{} $f$ is not regular if and only if $H_{A/(f)}(t)>(1-t^d)H_A(t)$. 
\end{enumerate}
\end{lemma}

The following lemma is well-known if $A=k[u_1, \dots, u_n]$.  

\begin{lemma} \label{lem.rHn}  Let $A$ be a locally finite $\NN$-graded algebra, and $f_i\in A_{d_i}, i=1, \dots, r$ homogeneous normal elements.  
Then $f_1, \dots, f_r$ is a regular sequence if and only if 
$$H_{A/(f_1, \dots, f_r)}(t)=(1-t^{d_1})\cdots (1-t^{d_r})H_A(t).$$
It follows that if $f_1, \dots, f_r$ is a regular sequence, then any permutation is also a regular sequence. 
\end{lemma} 

\begin{proof} 
If $f_1, \dots, f_r$ is a regular sequence, then 
$$H_{A/(f_1, \dots, f_{i-1}, f_{i})}(t)=(1-t^{d_i})H_{A/(f_1, \dots, f_{i-1})}(t)$$ 
for every $i=1, \dots , r$ by Lemma \ref{lem.nrHn} (1), so $$H_{A/(f_1, \dots, f_r)}(t)=(1-t^{d_1})\cdots (1-t^{d_r})H_A(t).$$

For the converse, define 
$$a_i(t):=\frac{1}{1-t^{d_1}}\cdots \frac{1}{1-t^{d_i}}H_{A/(f_1, \dots, f_i)}(t)$$
for $i=0, \dots, r$.  If $f_1, \dots, f_r$ is not a regular sequence, then $a_i(t)\geq a_{i-1}(t)$ for every $i=1, \dots, r$, and  $a_i(t)>a_{i-1}(t)$ for some $i=1, \dots, r$ by Lemma \ref{lem.nrHn} and Lemma \ref{lem.abc} (3) (4), so $a_0(t)\neq a_r(t)$ by Lemma \ref{lem.abc} (2), hence 
\begin{align*}
& H_{A/(f_1, \dots, f_r)}(t)= (1-t^{d_1})\cdots (1-t^{d_r})a_r(t) \\
\neq  &  (1-t^{d_1})\cdots (1-t^{d_r})a_0(t)=(1-t^{d_1})\cdots (1-t^{d_r})H_A(t). \qedhere
\end{align*}
\end{proof}

\section{Dualities}

The duality between $\sC_{n, 0}$ and $\sB_{n, n}$ 
% , which first appeared in \cite[Theorem 5]{S} (see also \cite[Proposition 7]{LeB}), 
is well-known (see \cite{AL}, \cite[Proposition 7]{LeB}).  This duality is even extended to graded skew Clifford algebras in \cite[Theorem 4.2]{CV1}.  
% It is very surprising to us
% that indeed 
In this section, we will show that there are dualities between $\sC_{n, r}$ and $\sB_{n, n-r}$ for all $r=0, \dots, n$.  To show it, we will explicitly compute the center of the graded Clifford algebra $C(F)$ in Section 3.3.
% which is essential to show the above duality.  
% , but not only for $r = 0$. 
% These dualities maybe known to experts. However, despite our best efforts to search relevant references, we could not find a single paper that clearly stated this result.
% Thus we give detailed proof in this paper. 

\subsection{Quadratic dual} \label{subsec.qud}

% In this subsection, we will show  that there is a natural bijection between $\sC_{n, 0}$ and $\sB_{n, n}$ as a preliminary result.  We first list some elementary results on vector spaces.  

% Let $V$ be a vector space (always finite dimensional!).  For a subspace $W\subset V$, we define $W^{\perp}:=\{\phi\in V^*\mid \phi(w)=0\; \forall w\in W\}$.   

% \begin{lemma} \label{lem.vww} 
% Let $V$ be a vector space and $W, W_1, W_2, W_3\subset V$ subspaces.
% \begin{enumerate}
% \item{} $\dim W^{\perp}=\dim V-\dim W$.  
% \item{} $(W_1^{\perp})^{\perp}=W_1$.  
% \item{} $(W_1+ W_2)^{\perp}=W_1^{\perp}\cap W_2^{\perp}$. 
% \item{} $(W_1\cap W_2)^{\perp}=W_1^{\perp}+W_2^{\perp}$. 
% \item{} If $V=W_1\oplus W_2\oplus W_3$, then $W_1^{\perp}=(W_1+ W_2)^{\perp}+(W_1+ W_3)^{\perp}$.  
% \end{enumerate}
% \end{lemma} 

% \begin{lemma} \label{lem.uvw} 
% Let $U, V$ be vector spaces and $W\subset U$ a subspace.  Viewing $W\otimes V$ as a subspace of $U\otimes V$, $(W\otimes V)^{\perp}=W^{\perp}\otimes V^*$.  
% \end{lemma} 

The notion of quadratic dual is useful to study quadratic algebras, in particular, Koszul algebras (known as Koszul duality). 

Let $V$ be a vector space (always finite dimensional!).  For a subspace $W\subset V$, we define $W^{\perp}:=\{\phi\in V^*\mid \phi(w)=0\; \forall w\in W\}$.   
For a quadratic algebra $A=T(V)/(R)$ where $R\subset V\otimes V$ is a subspace, we define the {\it quadratic dual} of $A$ by $A^!:=T(V^*)/(R^{\perp})$.  In the sequel, we often identify $V^{**}$ with $V$ and $R^{\perp\perp}$ with $R$, so that 
% By Lemma \ref{lem.vww} (2), 
$A\mapsto A^!$ gives a bijection from the set of isomorphism classes of quadratic algebras to itself such that $(A^!)^!=A$.

%%%%%%%%%%%%%%%%%%%%%%%%%%%%%%%%%%%
A connected graded algebra $A$ finitely generated in degree 1 over $k$ is called a {\it Koszul algebra} if the trivial module $k_A$ has a free resolution
$$
\xymatrix{
\cdots \ar[r] & P^i \ar[r] & \cdots \ar[r] & P^1 \ar[r] & P^0 \ar[r] & k_A \ar[r] & 0
}
$$
where $P^i$ is a graded free module generated in degree $i$ for each $i \geq 0$.

\begin{lemma}[{\cite[Theorem 5.9]{Sm}, \cite[Theorem 1.2]{ST}}] \label{lem.Kos}
 Let $A$ be a quadratic algebra. Then the following are equivalent.
\begin{enumerate}
\item{} $A$ is Koszul.
\item{} $A^!$ is Koszul.
\item{} $A^!$ is isomorphic to the Yoneda algebra $\uExt_A^*(k,k)$ as graded algebras.
\end{enumerate}
Further, if $A$ is Koszul, then the following holds.
\begin{enumerate}
\item[(i)] $H_{A^!}(t)=1/H_A(-t)$.  
\item[(ii)] $A/(f)$ is a Koszul algebra for every regular normal element $f\in A_2$.
\end{enumerate}
\end{lemma} 

We recall the following result.

\begin{lemma}[{\cite[Theorem 5.11]{Sm}}] \label{lem.qpaKos}
Every quantum polynomial algebra is Koszul. 
\end{lemma}

\begin{remark} \label{rem.BCKos}
By Lemma \ref{lem.rHn}, Lemma \ref{lem.Kos}, and Lemma \ref{lem.qpaKos}, 
every $A\in \sB_{n, r}\cup \sC_{n, r}$ is a Koszul algebra with $H_A(t)=(1-t^2)^r/(1-t)^n$. 
\end{remark}

Recall that a finite dimensional algebra $A$ is called a Frobenius algebra if there is a nondegenerate associative bilinear form $(-. -): A \times A \to k$. We will call a finite dimensional algebra $A$ {\it graded Frobenius} if $A$ is connected graded and Frobenius.  Note that a finite dimensional commutative algebra is Frobenius if and only if it is self-injective (Gorenstein of Krull dimension 0).

\begin{lemma} \label{lem.Sm} Let $S$ be a noetherian connected graded algebra.  Then $S$ is an $n$-dimensional quantum polynomial algebra if and only if $S^!$ is a graded Frobenius Koszul algebra such that $H_{S^!}(t)=(1+t)^n$. 
\end{lemma}  

\begin{proof} 
This follows from Lemma \ref{lem.Kos}, Lemma \ref{lem.qpaKos} and \cite[Theorem 5.10]{Sm}. 
\end{proof}

%%%%%%%%%%%%%%%%%%%%%%

\begin{definition} 
For a sequence $F=(F_1, \dots,  F_n)\in M_n(k)^{\times n}$ of symmetric matrices (not necessarily linearly independent), we define the graded algebra 
$$S^F := k\<x_1, \dots, x_n\>/(x_ix_j+x_jx_i-\sum _{m=1}^n(F_m)_{ij}x_m^2)_{1\leq i<j\leq n}.$$
We define the normalization $\overline F$ of $F$ by 
$$(\overline F_m)_{ij}=\begin{cases} 2 & \textnormal { if } i=j=m \\
0  & \textnormal { if } i=j\neq m \\
(F_m)_{ij}& \textnormal { if } i\neq j.\end{cases}$$ 
\end{definition} 

\begin{remark} Clearly, $S^F=S^{\overline F}$.  Again, the range of subscripts is important and sensitive in this paper. Since $F_m$ are symmetric for $1\leq m\leq n$, if $F$ is normalized, then we may write 
$$S^F := k\<x_1, \dots, x_n\>/(x_ix_j+x_jx_i-\sum _{m=1}^n(F_m)_{ij}x_m^2)_{1\leq i, j\leq n}.$$
\end{remark}

\begin{proposition} \label{prop.BSC} 
Let $F=(F_1, \dots,  F_n)\in M_n(k)^{\times n}$ be a sequence of symmetric matrices. 
\begin{enumerate}
\item{} $(S^F)^!=B(\overline F)$.  
\item{} If $F$ is normalized, then $B(F)^!=S^F$.
\item{} If 
$F$ is linearly independent and 
$C(F)$ is quadratic, 
then $B(F)^!\cong C(F)$.
\end{enumerate} 
\end{proposition} 

\begin{proof}   
(1) Write $S^F=T(V)/(R)$ where $R\subset V\otimes V$ and $B(\overline F)=T(V^*)/(W)$ where $W\subset V^*\otimes V^*$.  We denote by $\{x_1, \dots, x_n\}$ a basis of $V$, and by $\{u_1, \dots, u_n\}$ its dual basis of $V^*$.   Since $R$ is generated by linearly independent elements $\{x_ix_j+x_jx_i-\sum _{m=1}^n(F_m)_{ij}x_m^2\}_{1\leq i<j\leq n}$, we have $\dim R=n(n-1)/2$.  Since 
$$\{\sum_{1\leq i, j\leq n}(\overline F_m)_{ij}u_iu_j\}_{1\leq m\leq n}=\{2u_m^2+2\sum_{1\leq i<j\leq n}(F_m)_{ij}u_iu_j\}_{1\leq m\leq n},$$ 
and $W$ is generated by linearly independent elements 
$$
\{u_m^2+\sum_{1\leq i<j\leq n}(F_m)_{ij}u_i u_j\}_{1\leq m\leq n}\bigcup \{u_i u_j - u_j u_i\}_{1\leq i<j\leq n},
$$ 
we have $\dim W=n(n+1)/2$.  It is easy to check that $W\subset R^{\perp}$. Since 
$$\dim R^{\perp}=\dim (V\otimes V)-\dim R=n^2-n(n-1)/2=n(n+1)/2=\dim W, $$
%%% by Lemma \ref{lem.vww} (1), 
we have $R^{\perp}=W$, so 
$$(S^F)^!=T(V^*)/(R^{\perp})=T(V^*)/(W)=B(\overline F).$$

(2) If $F$ is normalized, then $\overline F=F$, so $B(F)^!=B(\overline F)^!= S^F$ by (1). 

(3) 
If $F$ is normalized, then $2x_j^2=\sum _{m=1}^n(F_m)_{jj}y_m=2y_j$ for every $j=1, \dots, n$ in $C(F)$.  Since $C(F)$ is quadratic, 
\begin{align*}
C(F) & =k\<x_1, \dots, x_n, y_1, \dots, y_n\>/(x_ix_j+x_jx_i-\sum _{m=1}^n(F_m)_{ij}y_m)_{1\leq i, j\leq n} \\
& =k\<x_1, \dots, x_n\>/(x_ix_j+x_jx_i-\sum _{m=1}^n(F_m)_{ij}x^2_m)_{1\leq i, j\leq n} \\
& =: S^F \cong B(F)^!
\end{align*}
by (2). In general, since $F$ is linearly independent, there exists a normalized $F'$ such that $F\sim _{st}F'$ by Lemma \ref{lem.clP2}, so $B(F)^!\cong B(F')^! = C(F')\cong C(F)$ by Lemma \ref{lem.bP} and Lemma \ref{lem.clP}. 
\end{proof} 

\begin{corollary}
If two graded Clifford algebras $C(F),C(F')$ are quadratic, then $C(F) \cong C(F')$ if and only if $F \sim_{st} F'$.
\end{corollary}

\begin{proof}
This follows from Lemma \ref{lem.bP} (3) and Proposition \ref{prop.BSC} (3).
\end{proof}

\subsection{Derivation quotient algebras} 

It was shown in \cite[Proposition 3.4]{H} that every 3-dimensional Clifford quantum polynomial algebra is a derivation quotient algebra of a symmetric superpotential.  This fact was essentially used to classify noncommutative conics in \cite{HMM}.  In this subsection, we will extend this fact to higher dimensional cases.

\begin{definition}
Let $V$ be a vector space, $W\subset V^{\otimes n}$ a subspace, $w\in V^{\otimes n}$, and $i\geq 0$.  
\begin{enumerate}
\item{} We define
\begin{align*}
& \partial W: = \{(\psi\otimes \id^{\otimes n-1})(w) \mid \psi \in V^*, w \in W \}, \\
& \partial^{i+1} W: =  \partial (\partial^{i} W), \\
& \cD(W,i) : = T(V)/(\partial^{i} W).
\end{align*}
\item{} 
In particular, we define
$$
\cD(w,i): = \cD(kw,i).  
$$
We call $\cD(w,i)$ the {\it $i$-th order derivation quotient algebra} of $w$.  
\item{} We call $w$ a {\it superpotential} if $\phi(w)=w$ where $\phi:V^{\otimes n}\to V^{\otimes n}$ is a linear map defined by 
$$\phi(v_1\otimes v_2\otimes \cdots v_{n-1}\otimes v_n)=v_n\otimes v_1\otimes \cdots \otimes v_{n-2}\otimes v_{n-1}.$$

\end{enumerate}
\end{definition}

Let $\frak S_n$ be the symmetric group of degree $n > 0$. Clearly, every element in 
$
{\rm Sym}^n(V) : = \{ w \in  V^{\otimes n}  \mid \theta (w) = w, \ \forall \theta \in \frak S_n \}
$ 
is a superpotential. The following result is a special case of \cite[Theorem 11]{DV}.

\begin{theorem}[Dubois-Violette]
Every $n$-dimensional quantum polynomial algebra is isomorphic to some $(n-2)$-order derivation quotient algebra. 
\end{theorem}

We use the following fact:

\begin{lemma} \label{lem.symn}  
If $W:={\rm Sym}^2(V)=\<u\otimes v+v\otimes u\mid u, v\in V\>\subset V\otimes V$, then 
$${\rm Sym}^n(V)=\bigcap _{i+j+2=n}V^{\otimes i}\otimes W\otimes V^{\otimes j}.$$
\end{lemma}

\begin{proposition} \label{prop.symd} 
Let $S=T(V)/(R)$ be a quantum polynomial algebra of dimension $n\geq 2$.  Then $S\in \sC_{n, 0}$ if and only if there exists $w\in {\rm Sym}^n(V)$ such that $S=\cD(w, n-2)$. 
\end{proposition} 

\begin{proof}   Let $W:={\rm Sym}^2(V)
\subset V\otimes V$.  By Corollary \ref{cor.clco}, $S\in \sC_{n, 0}$ if and only if $S^!=T(V^*)/(R^{\perp})$ is commutative if and only if $W^{\perp}\subset R^{\perp}$ if and only if  $R\subset W$. 

If 
$S\in \sC_{n, 0}$, then $R\subset W$, so $S=T(V)/(R)=\cD(w, n-2)$ for some
$$w\in \bigcap _{i+j+2=n}V^{\otimes i}\otimes R\otimes V^{\otimes j}\subset \bigcap _{i+j+2=n}V^{\otimes i}\otimes W\otimes V^{\otimes j}= {\rm Sym}^n(V)$$
by Lemma \ref{lem.symn}. 

Conversely, if $S=\cD(w, n-2)$ for some
$$w\in {\rm Sym}^n(V)=\bigcap _{i+j+2=n}V^{\otimes i}\otimes W\otimes V^{\otimes j}\subset V^{\otimes n-2}\otimes W$$
then $R=\partial ^{n-2}(kw)\subset W$, so $S\in \sC_{n, 0}$.
\end{proof} 

\begin{corollary} \label{cor.symd} $S\in \sC_{n, 0}$ is Calabi-Yau if and only if $n$ is odd. 
\end{corollary}

\begin{proof} By Proposition \ref{prop.symd}, $S=\cD(w, n-2)$ for some $w\in {\rm Sym}^n(V)$.  Since $w$ is a superpotential, the result follows from \cite[Theorem 6.3, Theorem 6.4]{MS}.   
\end{proof}  

\begin{remark}  $\sB_{n, 0}\subset \sA_{n, 0}$ for every $n$, however, by Corollary \ref{cor.symd},  $\sC_{n, 0}\subset \sA_{n, 0}$ if and only if $n$ is odd. 
\end{remark}

%%%\section{Clifford Quadratic Complete Intersections} 

\subsection{The center}  \label{center}

In this subsection, we explicitly compute the center of $C(F)$. We remind the reader that $k$ has characteristic $0$.  We also remind the reader that $uv+vu\in C(F)_2$ is a central element for every $u, v\in C(F)_1$ since it is a linear combination of central elements $y_1, \dots, y_n$.  
% A natural question is whether or not 
In this subsection, we will show that every central element $z \in C(F)_2$ is a linear combination of $y_1, \dots, y_n$. 
% The answer is yes, we will give a detailed proof in Section \ref{center} 
Note that, as mentioned in the proof of \cite[Theorem 5]{S} (see also Lemma \ref{lem.c1}), $y_1, \dots, y_n\in C(F)$ are linearly independent. 
% Also, $y_1, \dots, y_n$ are central in $C(F)$, and so are their linear combinations, e.g., for every $u,v \in C(F)_1$, $uv+vu$ is a linear combination of $y_m$, where $1 \leq m \leq n$, and hence central.  .}

\begin{lemma} \label{lem.bac2} 
Let ${R}$ be an integral domain, ${K}=Q({R})$ the field of fractions, and $\phi: {R}\to {K}$ the natural injection. 
For a symmetric matrix $\cF\in M_n({R})$,
$C_{R}(\cF)\otimes _{R}{K}\cong C_{K}(\phi(\cF))$ where $\phi(\cF):=(\phi(\cF_{ij}))\in M_n({K})$.
\end{lemma}  

\begin{proof}
By Lemma \ref{lem.bac} (2),  
\begin{align*}
C_R(\cF)\otimes _RK & :=R\<x_1, \dots, x_n\>/(x_ix_j+x_jx_i-\cF_{ij})
\otimes_RK \\
& \cong K\<x_1, \dots, x_n\>/(x_ix_j+x_jx_i-\phi(\cF_{ij})) \\
& = K\<x_1, \dots, x_n\>/(x_ix_j+x_jx_i-\phi(\cF)_{ij}) \\
& =: C_K(\phi(\cF)).  \qedhere
\end{align*}

\end{proof}

\begin{lemma} \label{lem.c4}  Let $K$ be a field, and $\cF\in M_n(K)$ a symmetric matrix.  
If $C_K(\cF)$ is a domain, then $\det \cF\neq 0$.  
\end{lemma}  

\begin{proof}
By Lemma \ref{lem.FF'} and Lemma \ref{lem.FF'2}, 
$$C_K(\cF)\cong C_K(\cF')= K\<x_1, \dots, x_n\>/(x_ix_j+x_jx_i, x_m^2-\l_m)_{1\leq i<j\leq n, 1\leq m\leq n}$$
for some diagonal matrix $\cF'={\rm diag}(\l_1, \dots, \l_n)$. If $\det \cF=0$, then $\det \cF'=0$, so $\l_i=0$ for some $i$.  It follows that $x_i^2=0$ in $C_K(\cF')$, so $C_K(\cF)$ is not a domain.  
\end{proof}

\begin{lemma} \label{lem.cRK} 
Let $R$ be an integral domain,  and $K:=Q(R)$ the field of fractions. 
If $A$ is a finitely generated $R$-algebra which is free as an $R$-module, and $\psi:A\to A\otimes _RK$ is the algebra homomorphism defined by  $\psi(a) = a\otimes 1$,
then $Z(A)=\psi^{-1}(Z(A\otimes _RK))$.  
\end{lemma}
 
\begin{proof} 
Left to the reader.
\end{proof}

The following result may be known to the experts, however, we will include our proof for the convenience of the reader.

\begin{theorem} \label{thm.Kev} 
Let $C(F)$ be a graded Clifford algebra, $R := k[y_1, \dots, y_n]$, and 
$\cF:=\sum _{m=1}^nF_my_m\in M_n(R)$. If $\det \cF\neq 0$, then 
$$Z(C(F))=\begin{cases} R & \textnormal { if $n$ is even,} \\
 R + Rg & \textnormal { for $g\in C(F)_n$ such that $g^2=\det \cF$ if $n$ is odd}. \end{cases}$$
\end{theorem}  

\begin{proof} %(Proof of Lemma \ref{lem.Kev})
By Lemma \ref{lem.c1} and Lemma \ref{lem.bac2}, we have an $R$-algebra homomorphism 
$$
\psi: C(F) \cong C_R(\cF) \to  C_R(\cF)\otimes _RK \cong C_K(\phi(\cF)); f \mapsto f\otimes 1
$$
where 
$K:=Q(R)$ is the field of fractions, 
and $\phi:R\to K$ is the natural injection.
By Lemma \ref{lem.FF'2}, there exists $P=(p_{ij})\in \GL_n(K)$ such that $\cF':=P^t\phi(\cF)P$ is a diagonal matrix.  
By  Lemma \ref{lem.FF'}, the $K$-algebra automorphism $\phi_P$ of $K\<x_1, \dots, x_n\>$ defined by $\phi_P(x_j)=\sum^n_{i = 1} p_{ij}x_i$ induces a $K$-algebra isomorphism $
\phi_P:C_K(\cF')\to C_K(\phi(\cF))$
so that we have the following commutative diagram$$
\xymatrix{
K \langle x_1, \dots, x_n\rangle \ar[r]^-{\phi_P}_-{\cong} \ar[d]_-{\pi'} & K \langle x_1, \dots, x_n\rangle \ar[d]^-{\pi} \\
C_K(\cF') \ar[r]^-{\phi_P}_-{\cong} & C_K(\phi(\cF))
}
$$
where $\pi: K \langle x_1, \dots, x_n\rangle \to C_K(\phi(\cF))$, $\pi': K \langle x_1, \dots, x_n\rangle \to C_K(\cF')$ are natural surjections.

If $w : = \sum_{\sigma \in \frak S_n} (\sgn \sigma) x_{\sigma(1)} \cdots x_{\sigma(n)} \in K \langle x_1, \dots, x_n\rangle_n$, then it is well-known that 
$\phi_P(w) = (\det P) w$, so 
$$\phi_P(\pi'(w)) = \pi(\phi_P(w)) = \pi((\det P) w) = (\det P) \pi(w).$$ 
Since $\cF'$ is a diagonal matrix, $x_j x_i = - x_i x_j$ in $C_K(\cF')$ for every $i < j$, so $x_{\sigma(1)} \cdots x_{\sigma(n)} = (\sgn \sigma)x_1 \cdots x_n$  in $C_K(\cF')$ for every $\s \in  \frak{S}_n$. It follows that 
$$
\Phi_P^{-1}(\pi(w)) = \pi'(w)/\det P = |\frak{S}_n| x_1 \cdots x_n / \det P \in K x_1 \cdots x_n,
$$
so
$$
Z(C_K(\phi(\cF))) = \begin{cases} K & \text{if $n$ is even},\\ K + K \pi(w) & \text{if $n$ is odd} \end{cases}
$$
by Lemma \ref{lem.cKf}. Let $g := \sum_{\sigma \in \frak S_n} (\sgn \sigma) x_{\sigma(1)} \cdots x_{\sigma(n)} \in C(F)_n$.  Since $\psi(g) = \pi(w)$ in $C_K(\phi(\cF))$, $R + Rg \subset \psi^{-1}(K + K \pi(w))$. 
By Lemma \ref{lem.HO}, $C(F) \cong C_R(\cF)$ is free as an $R$-module with a basis 
$$
\mathcal{B}: = \{ 1, x_{i_1} \cdots x_{i_t} \mid 1 \leq t \leq n, i_1 < \dots < i_t\},
$$ and $C_K(\phi(\cF))$ is free as a $K$-module with a basis $\psi(\mathcal{B})$. 
Since $x_jx_i = -x_i x_j + \cF_{ij}$ in $C(F)$ and $\cF_{ij}\in R\subset Z(C(F))$ for every $i < j$, the difference $x_{\sigma(1)} \cdots x_{\sigma(n)} - (\sgn \sigma) x_1 \cdots x_n\in C(F)$ is an $R$-linear combination of elements in $\mathcal{B}\setminus \{x_1 \cdots x_n \}$ for every $\s \in  \frak{S}_n$.  It follows that  
$
g- |\frak S_n | x_{1} \cdots x_{n} \in C(F)
$ 
is an $R$-linear combination of elements in $\mathcal{B}\setminus \{x_1 \cdots x_n \}$, so  
$$
\mathcal{B}' : = (\mathcal{B}\setminus \{x_1 \cdots x_n \}) \cup \{ g \} = \{ b_0 :=1, b_1 :=g ,b_2 ,\dots, b_{2^n-1}\}
$$ 
is an $R$-basis of $C(F)$, and  $\psi(\mathcal{B}')$ is an $K$-basis of $C_K(\phi(\cF))$. If $f  = \sum_{i=0}^{2^n-1} r_i b_i\in C(F)$ where $r_i \in R$ such that $\psi(f) \in K + K \pi(w)$, then $r_i = 0$ for $2\leq i \leq 2^n-1$, so $f\in R+Rg$.  It follows that $\psi^{-1}(K+ K \pi(w)) = R + Rg$, so 
$$
Z(C(F))=\begin{cases} R & \textnormal { if $n$ is even,} \\
 R + Rg & \textnormal { if $n$ is odd} \end{cases}
$$
by Lemma \ref{lem.cRK}.

Since 
\begin{align*}
(x_1 \cdots x_n)^2 =& (-1)^{\frac{n(n-1)}{2}} x_1^2 \cdots x_n^2 = (-1)^{\frac{n(n-1)}{2}} (\cF'_{11}/2) \cdots (\cF'_{nn}/2) \\
=&(-1)^{\frac{n(n-1)}{2}} \det \cF'/2^n
\end{align*}
in $C_K(\cF')$, 
\begin{align*}
\psi(g^2) =& \psi(g)^2= \pi(w)^2 = \phi_P(|\frak{S}_n|x_1 \cdots x_n /\det P)^2 \\
=&(n!)^2 \phi_P((x_1 \cdots x_n)^2)/(\det P)^2 = (-1)^{\frac{n(n-1)}{2}} \frac{(n!)^2}{2^n } \det \cF' / (\det P)^2\\
=& (-1)^{\frac{n(n-1)}{2}} \frac{(n!)^2}{2^n } \det (P^t \phi(\cF) P)/(\det P)^2 = (-1)^{\frac{n(n-1)}{2}} \frac{(n!)^2}{2^n} \det \phi(\cF)\\
=& (-1)^{\frac{n(n-1)}{2}} \frac{(n!)^2}{2^n} \det \cF \in R,
\end{align*}
that is, $g^2 = (-1)^{\frac{n(n-1)}{2}} \frac{(n!)^2}{2^n} \det \cF$ in $C(F)$. By adjusting $g$ by a scalar, we may take $g \in C(F)_n$ such that $g^2 = \det \cF$.
Note that if $n$ is odd, then $Z(C(F))_n=R_0g=kg$, so if $g'=a g\in Z(C(F))_n$ where $a\in k$ such that ${g'}^2=\det \cF=g^2$, then $g'=\pm g$.
\end{proof}

\begin{remark}  The above proof is much simpler and the non-trivial central element $g$ is more explicit compared, for example, to those in \cite{Se}.
\end{remark}

The following explicit description of the center can be regarded as a generalization of \cite[Lemma 3.6]{H}, which was essentially used to classify noncommutative conics in \cite{HMM}.  Note again that $k$ has characteristic $0$.

\begin{corollary} \label{cor.center}
If $S=C(F)\in \sC_{n, 0}$ for some normalized $F$, then 
 $Z(S)_2=\sum _{m=1}^nkx_m^2$.    
 \end{corollary} 

 \begin{proof}  
% {\tcr By \cite{AL},}  
% Corollary \ref{cor.cbbd}, 
Since a graded Clifford algebra $C(F)$ is the enveloping algebra of a {\it $(n,n)$-quadratic graded Lie algebra} (\cite[Definition 1.2]{AL}), 
% the results follows from 
$C(F)$ is a domain by \cite{AL}.
By Lemma \ref{lem.c1} and Lemma \ref{lem.bac2}, $C(F)\otimes _RK\cong C_R(\cF)\otimes _RK\cong C_K(\phi(\cF))$ is a domain where $R=k[y_1, \dots, y_n]$, $K=Q(R)$, $\phi:R\to K$ is the natural injection, and $\cF:=\sum _{m=1}^nF_my_m\in M_n(R)$, so $\det \phi(\cF)\neq 0$ by Lemma \ref{lem.c4}, hence $\det \cF\neq 0$.  Since $F$ is normalized, we have $y_m=x_m^2$ for $1\leq m\leq n$, so 
$$Z(S)_2=Z(C(F))_2 =R_2=\sum _{m=1}^nky_m=\sum _{m=1}^nkx_m^2$$
by Theorem \ref{thm.Kev}.   
\end{proof}

\subsection{Duality}

In this subsection, we will show 
that there are dualities between $\sC_{n, r}$ and $\sB_{n, n-r}$ for all $r=0, \dots, n$. 
The duality for $r=0$ is well-known (which, in particular, implies Theorem \ref{thm.1} in Introduction).  We give here a short proof for the convenience of the reader.

\begin{theorem} \label{thm.n0nn} 
The map $(-)^!:\sC_{n, 0}\to \sB_{n, n}$ is a bijection. 
\end{theorem} 

\begin{proof}
If $S\in \sC_{n, 0}$, then we may assume that $S=C(F)$ for some normalized $F$ by Corollary \ref{cor.nor}.   Since $S$ is quadratic, $S=C(F)\cong B(F)^!$ by Proposition \ref{prop.BSC} (3), so $S^!\cong B(F)$ is commutative. 
Since $S^!$ is quadratic and 
$$
H_{S^!}(t)=1/H_S(-t)=(1+t)^n=(1-t^2)^n/(1-t)^n
$$ 
by Lemma \ref{lem.Kos} (i), $S^!\cong k[u_1, \dots, u_n]/(f_1, \dots, f_n)$ for some $f_1, \dots, f_n\in k[u_1, \dots, u_n]_2$, so $S^!\in \sB_{n, n}$ by Lemma \ref{lem.rHn}.  
Thus the map $(-)^!:\sC_{n, 0}\to \sB_{n, n}$ is well-defined. 

It suffices to show that the map $(-)^!:\sB_{n, n} \to \sC_{n, 0}$ is also well-defined, and this
follows from \cite[Theorem 5]{S} (see also \cite[Proposition 7]{LeB}). 
\end{proof}

The following corollary is an interesting observation. 

\begin{corollary} \label{cor.clco} 
Let $S$ be a quantum polynomial algebra.  Then $S\in \sC_{n, 0}$ 
if and only if $S^!$ is commutative. 
\end{corollary}  

% \begin{corollary} \label{cor.cbbd}
% For a sequence $F= (F_1, \dots, F_n)\in M_n(k)^{\times n}$ of linearly independent symmetric matrices, the following are equivalent:
% \begin{enumerate}
% \item{} $C(F)\in \sC_{n, 0}$.
% \item{} $B(F)\in \sB_{n, n}$.  
% \item{} $C(F)$ is a domain. 
% \end{enumerate}
% \end{corollary} 

% \begin{proof} 
% (1) $\Leftrightarrow$ (2): By  Theorem \ref{thm.n0nn}. 
% (2) $\Leftrightarrow$ (3): Since a graded Clifford algebra $C(F)$ is the enveloping algebra of a {\it $(n,n)$-quadratic graded Lie algebra} (\cite[Definition 1.2]{AL}), the results follows from \cite{AL}.
% \end{proof} 

To show the duality for $r\geq 1$, the following proposition is essential.
%for the proof.  
Since not every sequence of linearly independent elements $f_1, \dots, f_r\in k[u_1, \dots, u_n]_2$ forms a regular sequence, it is rather surprising.  

\begin{proposition} \label{prop.sq} Let $S\in \sC_{n, 0}$.  
If $g_1, \dots , g_r\in Z(S)_2$ are linearly independent elements, then they form a regular sequence so that  $S/(g_1, \dots, g_r)\in \sC_{n, r}$.  In particular, if $g_1, \dots , g_n\in Z(S)_2$ are linearly independent elements, then $S/(g_1, \dots, g_n)$ is the exterior algebra.
\end{proposition} 

\begin{proof} 
By Corollary \ref{cor.nor}, we may assume that $S=S^F$ for some normalized $F$.  If $g_1, \dots , g_n\in Z(S)_2$ are linearly independent elements, then $\sum _{m=1}^nkg_m=\sum _{m=1}^nkx_m^2$, so $S/(g_1, \dots, g_n)=S^F/(x_1^2, \dots, x_n^2)$ is the exterior algebra.    If $g_1, \dots , g_r\in Z(S)_2$ are linearly independent elements, then there exist $g_{r+1}, \dots, g_n\in Z(S)_2$ such that 
 $\sum _{m=1}^nkg_m=\sum _{m=1}^nkx_m^2$, so 
 $$H_{S/(g_1, \dots , g_n)}(t)=(1+t)^n=(1-t^2)^n/(1-t)^n=(1-t^2)^nH_S(t).$$ 
By Lemma \ref{lem.rHn}, $g_1, \dots , g_n\in Z(S)_2$ is a regular sequence, so $g_1, \dots , g_r\in Z(S)_2$ is a regular sequence. 
\end{proof} 

\begin{corollary}  \label{cor.sq}
The map $(-)^!:\sC_{n, n}\to \sB_{n, 0}$ is well-defined. 
\end{corollary} 

\begin{proof} For every $A=S/(g_1, \dots, g_n)\in \sC_{n , n}$ where $S\in \sC_{n, 0}$ and $g_1, \dots, g_n\in Z(S)_2$ is a regular sequence, since $g_1, \dots, g_n$ are linearly independent, $A$ is the exterior algebra by Proposition \ref{prop.sq}, so $A^!\cong k[x_1, \dots, x_n]\in \sB_{n, 0}$. 
\end{proof}

% \begin{remark} \label{rem.ce2}
Let $A=T(V)/(R)$ be a quadratic algebra where $R\subset V\otimes V$ is a subspace.  Denote the quotient map $T(V)\to A$ by $u\mapsto \overline u$.  For a regular element $f\in Z(A)_2$, choose $w\in T(V)_2=V\otimes V$ such that $\overline w=f\in Z(A)_2$, and choose a subspace $W\subset V\otimes V$ such that $V\otimes V=kw\oplus R\oplus W$. Then we can show that  there exists $w^!\in V^*\otimes V^*$
such that 
$kw^!=
(R+W)^{\perp}=R^{\perp}\cap W^{\perp}$.
% Define $f^!: =\overline {w^!}\in (A/(f))^!_2$, then $A^!= (A/(f))^!/(f^!)$.}
% For $u\in T(V)_3=V\otimes V\otimes V$, $\overline u=0$ in $A_3=V^{\otimes 3}/(R\otimes V+V\otimes R)$ if and only if 
% $u\in R\otimes V+V\otimes R$.  It follows that, for $w\in T(V)_2=V\otimes V$, $\overline w\in Z(A)_2$ if and only if 
% $w\otimes v-v\otimes w\in R\otimes V+V\otimes R$ for every $v\in V$.  
% \end{remark}  

% The following result was proved in \cite[Section 4]{HY}.  We reproduce the proof since we need to recall how to construct $f^!$ from $f$.  

\begin{lemma}[{\cite[Section 4]{HY}, \cite[Corollary 1.4 (1)]{ST}}] \label{lem.ffs}  
Keeping the notations above, $f^!: =\overline {w^!}\in Z((A/(f))^!)_2$ such that $A^!= (A/(f))^!/(f^!)$. 
% We have $f^! \in Z((A/(f))^!))$.
Moreover, if $A$ is Koszul, 
then $f^!\in (A/(f))^!$ is regular.  
\end{lemma}

\begin{lemma} \label{lem.f1r} 
Let $A$ be a Koszul algebra.  For every central regular sequence $f_1, \dots, f_r\in Z(A)_2$, there exists a central regular sequence $f_1^!, \dots, f_r^!\in Z((A/(f_1, \dots, f_r))^!)_2$ such that $A^!=(A/(f_1, \dots, f_r))^!/(f_1^!, \dots, f_r^!)$.   
\end{lemma}  

\begin{proof} We prove by induction on $r$.  The result holds for $r=1$ by Lemma \ref{lem.ffs}.   Suppose that the result holds for $r-1$.  Let $A':=A/(f_1)$, $A''=A/(f_2, \dots, f_r)$ and $B:=(A/(f_1, \dots, f_r))^!=(A'/(f_2, \dots, f_r))^!=(A''/(f_1))^!$.  By Lemma \ref{lem.Kos} (ii), $A', A''$ are Koszul algebras.  

Since $f_1\in Z(A)_2$ is regular, there exists a regular element $f_1^!\in Z((A')^!)_2$ such that $A^!=(A')^!/(f_1^!)$ by Lemma \ref{lem.ffs}.  Since 
$f_2, \dots, f_r\in Z(A')_2$ is a central regular sequence,  there exists a central regular sequence $f_2^!, \dots, f_r^!\in Z(B)_2$ such that $(A')^!=B/(f_2^!, \dots, f_r^!)$ by induction.  

On the other hand, since 
$f_2, \dots, f_r, f_1\in Z(A)_2$ is a central regular sequence by Lemma \ref{lem.rHn}, $\bar f_1\in Z(A'')_2$ is regular, so there exists $(\bar f_1)^!\in Z(B)_2$ such that $(A'')^!= B/((\bar f_1)^!)$ by Lemma \ref{lem.ffs}. 

Let $A=T(V)/(R)$ and $A''=T(V)/(R\oplus R')$ where $R, R'\subset V\otimes V$ are subspaces, 
and choose $w\in V\otimes V$ such that $\overline w=f_1\in A$ and a subspace $W\subset V\otimes V$ such that $V\otimes V=kw\oplus R\oplus R'\oplus W$.  By construction of $f_1^!$ and $(\bar f_1)^!$ in 
% the proof of 
Lemma \ref{lem.ffs}, 
$$\overline {(\bar f_1)^!}=\overline {w^!}=f_1^!\in Z((A')^!)_2=Z(B/(f_2^!, \dots, f_r^!))_2$$ 
is regular where $w^!\in V^*\otimes V^*$ such that $kw^!=(R\oplus R'\oplus W)^{\perp}$,  
so  $f_2^!, \dots, f_r^!, {\bar f_1}^!, \in Z(B)_2$ is a central regular sequence such that
\begin{align*}
A^! & =(A')^!/(f_1^!)=(B/(f_2^!, \dots, f_r^!))/(f_1^!) \\
& =B/(f_2^!, \dots, f_r^!, (\bar f_1)^!)=(A/(f_1, \dots, f_r))^!/((\bar f_1)^!, f_2^!, \dots, f_r^!). \qedhere
\end{align*}
\end{proof}

The following diagram may be useful to understand the above proof.  
$$\begin{array}{ccccccccccc}
f_1 & \in & A & \twoheadrightarrow & A' & & A^! &  \twoheadleftarrow & (A')^! & \ni & f_1^! \\
\downarrow & & \rotatebox{90}{$\twoheadleftarrow$} &  & \rotatebox{90}{$\twoheadleftarrow$} & & \rotatebox{90}{$\twoheadrightarrow$} & & \rotatebox{90}{$\twoheadrightarrow$} & & \uparrow \\
\bar f_1 & \in & A'' & \twoheadrightarrow & B^! & & (A'')^! & \twoheadleftarrow & B & \ni & (\bar f_1)^!
\end{array}$$

\begin{remark} \label{rem.cseq} 
In the literature, it may be more standard to define a central sequence $f_1, \dots, f_r\in A$ by the condition $\bar {f_i}\in Z(A/(f_1, \dots, f_{i-1}))$ for every $i=1, \dots, r$ (eg. \cite[Definition 1.9]{CV1}, \cite{LeB}).   If we adopt this definition, then the above lemma ``Let $A$ be a Koszul algebra.  For every central regular sequence $f_1, \dots, f_r\in A_2$, there exists a central regular sequence $f_r^!, \dots, f_1^!\in (A/(f_1, \dots, f_r))^!_2$ such that $A^!=(A/(f_1, \dots, f_r))^!/(f_r^!, \dots, f_1^!)$'' can be proved by a simple induction of Lemma \ref{lem.ffs}.   In this case, the order of the sequence may be important.  For example, if 
$$A=k\<x, y, z\>/(yz-zy-x^2, zx-xz, xy-yx),$$ 
then $x^2, y^2\in A_2$ is a central regular sequence, but $y^2, x^2\in A_2$ is not by the above definition.   Our definition of a central sequence, which requires $f_1, \dots , f_r\in Z(A)$, is useful since then a central regular sequence is preserved by any permutation (see Lemma \ref{lem.rHn}).  Moreover, if $A\in \sC_{n, 0}$, 
%  is a Clifford quantum polynomial algebra, 
then any subsequence of a central regular sequence $f_1, \dots, f_r\in Z(A)_2$ is also a central regular sequence by Proposition \ref{prop.sq} by our definition, which is useful in the proof of the next theorem.  
\end{remark}

% We remind the reader again that t
The following theorem may be known by experts,  but, to our knowledge, it has not appeared in the literature except for the case $r=0, n$. 
% Due to its usefulness in the classification problem, 
Since there are some delicate issues (see Remark \ref{rem.cseq}, Remark \ref{rem.cseq2}), we give a detailed proof in this paper. 
% the following. 

% The following theorem is a main result of the paper.

\begin{theorem} \label{thm.sq} 
The map $(-)^!:\sC_{n, r}\to \sB_{n, n-r}$ is a bijection for every $r=0, \dots, n$.  
\end{theorem} 

\begin{proof}
It is enough to show that the map 
$(-)^!:\sB_{n, n-r}\to \sC_{n, r}$ and
the map 
$(-)^!:\sC_{n, r}\to \sB_{n, n-r}$ are well-defined
for every  $0\leq r\leq n$.   

For every $B\in \sB_{n, n-r}$, there exists a regular sequence $f_1, \dots, f_r\in B_2$ such that $B/(f_1, \dots, f_{r})\in \sB_{n. n}$,  so $S:=(B/(f_1, \dots, f_r))^!
\in \sC_{n, 0}$ by Theorem \ref{thm.n0nn}.  
By Lemma \ref{lem.f1r},  there exists a central regular sequence $f_1^!, \dots, f_r^!\in Z(S)_2$ such that 
$B^! =
S/(f_1^!, \dots, f_r^!)\in \sC_{n, r}$, 
so 
$(-)^!:\sB_{n, n-r}\to \sC_{n, r}$ is well-defined for every $0\leq r\leq n$. 
 
Similarly, for every $A=S/(g_1, \dots, g_r)\in \sC_{n, r}$ where $S\in \sC_{n, 0}$ and $g_1, \dots, g_r\in Z(S)_2$ is a central regular sequence, there exists a central regular sequence $g_1, \dots, g_r, g_{r+1}, \dots, g_n\in Z(S)_2$ such that $A/(g_{r+1}, \dots, g_{n})=S/(g_1, \dots, g_r, g_{r+1}, \dots, g_n)\in \sC_{n. n}$ by Proposition \ref{prop.sq},  so $(A/(g_{r+1}, \dots, g_n))^!=k[u_1, \dots, u_n]
\in \sB_{n, 0}$ by Corollary \ref{cor.sq}.  
Since $g_{r+1}, \dots, g_n\in Z(A)_2$ is a central regular sequence by Proposition \ref{prop.sq}, there exists a regular sequence $g_{r+1}^!, \dots, g_n^!\in k[u_1, \dots, u_n]_2$ such that 
$A^! =
k[u_1, \dots, u_n]/(g_{r+1}^!, \dots, g_n^!)\in \sB_{n, n-r}$ by Lemma \ref{lem.f1r}, 
so $(-)^!:\sC_{n, r}\to \sB_{n, n-r}$ is well-defined for every $0\leq r\leq n$.  
\end{proof}

\begin{remark} \label{rem.cseq2}
In addition to our nonstandard definition of a central regular sequence (Remark \ref{rem.cseq}), another key point to prove the above theorem is that we need to show that there is no central elements of degree 2 in $C(F) \in \sC_{n, 0}$ other than $\sum_{m=1}^nky_m$, to make sure that  the exterior algebra is the only algebra in $\sC_{n, n}$ (see  Corollary \ref{cor.sq}). 
That is why we computed the center of $C(F)$ explicitly in Section 3.3. 
% Corollary \ref{cor.center}.     
\end{remark}

\begin{remark} \label{rem.Sesy} 
{\it Segre symbols} are classical notions which are useful to classify algebras of the form $k[x_1, \dots, x_n]/(f_1,f_2)$ where $f_1,f_2 \in k[x_1, \dots, x_n]_2$ (see \cite[Section XII]{HP}). For $B, B'\in \sB_{n, 2}$, if $B\cong B'$, then they have the same Segre symbol, and the converse holds for $n=3$.  
It is known that there are exactly $6$ distinct Segre symbols for $\sB_{3,2}$ (see Table \ref{tab.pa} of Section 5), so $\#(\sB_{3, 2})=6$.
\end{remark}

\begin{corollary} \label{cor.123}
$\#(\sA_{3, 1})=9, \#(\sA_{3, 2})=9, \#(\sA_{3, 3})=\infty$.  
\end{corollary} 

\begin{proof} By \cite [Corollary 3.8]{HMM} and \cite[Corollary 4.8]{SV}, $\sA_{3, 1}$ is the disjoint union of $\sB_{3, 1}$ and $\sC_{3, 1}$.  If $A=S/(f_1, f_2)=(S/(f_1))/(f_2)\in \sA_{3, 2}$ where $S\in \sA_{3, 0}$ and $f_1, f_2\in Z(S)_2$ is a central regular sequence, then either $S/(f_1)\in \sB_{3, 1}$ or $S/(f_1)\in \sC_{3, 1}$, so either $A\in \sB_{3, 2}$ or $A\in \sC_{3, 2}$.  On the other hand, if $A\in \sC_{3, 2}$, then there exists a central regular sequence $f_1, f_2, f_3\in Z(S)_2$ such that $A/(f_3)=S/(f_1, f_2, f_3)$ is the exterior algebra by Proposition \ref{prop.sq}, so $A\not \in \sB_{3, 2}$,
hence $\sA_{3, 2}$ is the disjoint union of $\sB_{3, 2}$ and $\sC_{3, 2}$.  Similarly, we can show that $\sA_{3, 3}$ is the disjoint union of $\sB_{3, 3}$ and $\sC_{3, 3}$.  

Since $\#(\sB_{3, 0})=1$, $\#(\sB_{3, 1})=3$ (Lemma \ref{lem.sy11}), $\#(\sB_{3, 2})=6$ (Remark \ref{rem.Sesy}), and $\#(\sC_{3, 0})=\infty$ (Example \ref{exm.3-dimgcliff}),  
\begin{align*}
& \#(\sA_{3, 1})=\#(\sB_{3, 1})+\#(\sC_{3, 1})=\#(\sB_{3, 1})+\#(\sB_{3, 2})=3+6=9, \\
& \#(\sA_{3, 2})=\#(\sB_{3, 2})+\#(\sC_{3, 2})=\#(\sB_{3, 2})+\#(\sB_{3, 1})=6+3=9, \\
& \#(\sA_{3, 3})=\#(\sB_{3, 3})+\#(\sC_{3, 3})=\#(\sB_{3, 3})+\#(\sB_{3, 0})=\infty+1=\infty 
\end{align*}
by Theorem \ref{thm.sq}.  
\end{proof}

% {\tcr We give a list of $\sA_{3, 1}$ in Table \ref{tab-a31}. }

% \begin{table}[h]
% \caption{List of $\mathscr{A}_{3, 1}$.} \label{tab-a31}
% \begin{tabular}{|ccc|} 
% \hline
\begin{example}
A complete representatives of algebras in $\mathscr{A}_{3, 1}$ is given below (see \cite[Corollary 3.8]{HMM}):
$$\begin{array}{lll}
k[x, y, z]/(x^2), & k[x, y, z]/(x^2+y^2), & k[x, y, z]/(x^2+y^2+z^2), \\ 
S^{(0, 0, 0)}/(x^2), & S^{(0, 0, 0)}/(x^2+y^2), & S^{(0, 0, 0)}/(x^2+y^2+z^2), \\ 
S^{(1, 0, 0)}/(y^2), & S^{(1, 0, 0)}/(x^2+y^2 + z^2), &  S^{(1, 1, 0)}/(3x^2+3y^2+4z^2), 
\end{array}$$
% \end{tabular}
% \begin{tablenotes}
% \item 
where $S^{(a,b,c)} = k \langle x,y,z \rangle /(yz + zy + a x^2, zx + xz + b y^2, xy + yx + c z^2)$.
\end{example} 
% \end{tablenotes}
% \end{table}

\begin{remark} 
$\sA_{3, 0}$ is a far bigger than the union of $\sB_{3, 0}$ and $\sC_{3, 0}$ (see \cite[Table 1]{HMM}). 
\end{remark}

\section{Characteristic Varieties} 

In this section, we define the characteristic varieties $X^{(s)}_A$ for $A\in \sC_{n, r}$.  We show that $A$ has a point variety $E_A$ and there exists a double cover $E_A\to X^{(3)}_A$.  

\subsection{Clifford quantum polynomial algebras}

In this subsection, we first show some geometric properties of Clifford quantum polynomial algebras. 

\begin{definition} \label{defn.CV}
For a sequence  $F=(F_1, \dots, F_r)\in M_n(k)^{\times r}$ of symmetric matrices so that $\cF:=F_1y_1+\cdots+F_ry_r\in M_n(k[y_1, \dots, y_r])$, and $1\leq s\leq n$, 
we define 
\begin{align*}
X^{(s)}(F) & :=
\Proj k[y_1, \dots, y_r]/(\textnormal {all the $s$-minors of $\cF$
}) \subset \PP^{r-1} \\
& =\{(\l_1, \dots, \l_r)\in \PP^{r-1}\mid \rank (\l_1F_1+\cdots +\l_rF_r)< s\}.
\end{align*} 
\end{definition}

\begin{lemma} \label{lem.XFB} 
Let $F=(F_1, \dots, F_r), F'=(F'_1, \dots, F'_r)\in M_n(k)^{\times r}$
be sequences of symmetric matrices.
\begin{enumerate}
\item{} If $F\sim _tF'$, then $X^{(s)}(F)=X^{(s)}(F')$.  
\item{} If $F\sim _sF$, then $X^{(s)}(F)\cong _pX^{(s)}(F')$ (projectively equivalent).
\item{} In particular, if $B(F)\cong B(F')$, then  $X^{(s)}(F)\cong _pX^{(s)}(F')$.
\end{enumerate}
\end{lemma}

\begin{lemma} \label{lem.liX1} \label{lem.x3F}
Let $F=(F_1, \dots, F_r)\in M_n(k)^{\times r}$
be a sequence of symmetric matrices, and $f_m:=\sum_{1\leq i, j\leq n}(F_m)_{ij}u_iu_j\in k[u_1, \dots, u_n]_2$ for $1\leq m\leq r$.  
\begin{enumerate}
\item{} $F$ is linearly independent if and only if $X^{(1)}(F)=\emptyset$.
\item{} $(\l_1, \dots, \l_r)\in X^{(2)}(F)$ if and only if there exists $0\neq g\in k[u_1, \dots, u_n]_1$ such that $g^2=\sum _{m=1}^r\l_mf_m$. 
\item{} $(\l_1, \dots, \l_r)\in X^{(3)}(F)$ if and only if there exist $0\neq g_1, g_2\in k[u_1, \dots, u_n]_1$ such that $g_1g_2=\sum _{m=1}^r\l_mf_m$. 
\end{enumerate}
\end{lemma} 

\begin{proof} (1)  $F$ is linearly independent if and only if $\sum_{m=1}^r\l_mF_m=0$ implies $(\l_1, \dots, \l_r)=(0, \dots, 0)$ if and only if  $\rank (\sum_{m=1}^r\l_mF_m)=0$ implies  $(\l_1, \dots, \l_r)=(0, \dots, 0)$
if and only if $X^{(1)}(F)=\emptyset$. (2) and (3) are clear.
\end{proof}

For $p=(p_1, \dots, p_n)\in \AA^n$, we write $g_p:=\sum _{i=1}^np_iu_i\in k[u_1, \dots, u_n]_1$.  
For $p=(p_1, \dots, p_n), q=(q_1, \dots, q_n)\in \AA^n$, we write 
$p*q=(p_1q_1, \dots, p_nq_n)\in \AA^n$.  We often view $p\in \PP^{n-1}$ by abuse of notation.  
 Note that if $\G\subset \PP^{n-1}\times \PP^{n-1}$ is a subvariety such that, for every $(p, q)\in \G$,  $(p_1q_1, \dots, p_nq_n)\neq (0, \dots, 0)$, then the map 
 $$
 \G\to \PP^{n-1}; \; (p, q)\mapsto p*q
 $$ 
 is a well-defined morphism of varieties (by abuse of notation).

\begin{theorem} \label{thm.rXF}
If $S=T(V)/(R)=S^F\in \sC_{n, 0}$ for some normalized $F$,   
then $\Psi:\cV(R)\to X^{(3)}(F): (p, q)\mapsto p*q$ is (at most) a double cover.  
\end{theorem}  

\begin{proof} Let $f_m:=\sum_{1\leq i, j\leq n}(F_m)_{ij}u_iu_j\in k[u_1, \dots, u_n]_2$ so that $S^!=B(F)=k[u_1, \dots, u_n]/(f_1, \dots, f_n)$ by Proposition \ref{prop.BSC} (2).  For $p=(p_1, \dots, p_n), q=(q_1, \dots, q_n)\in \PP^{n-1}$, if $(p, q)\in \cV(R)$, then
$$p_iq_j+p_jq_i=\sum _{m=1}^n(F_m)_{ij}p_mq_m$$
for every $1\leq i<j\leq n$.  In $k[u_1, \dots, u_n]$, 
\begin{align*}
0\neq g_pg_q & :=\left (\sum_{i=1}^n p_iu_i\right )\left (\sum _{j=1}^nq_ju_j\right ) =\sum_{1\leq i, j\leq n}p_iq_ju_iu_j \\ & =\sum_{m=1}^np_mq_mu_m^2+\sum_{1\leq i<j\leq n}(p_iq_j+p_jq_i)u_iu_j \\
& =\sum_{m=1}^np_mq_mu_m^2+\sum _{1\leq i<j\leq n}\left(\sum_{m=1}^n(F_m)_{ij}p_mq_m\right)u_iu_j \\
& =\sum_{m=1}^np_mq_m\left (\frac{1}{2}\sum _{1\leq i, j\leq n}(F_m)_{ij}u_iu_j\right) =\frac{1}{2}\sum_{m=1}^np_mq_mf_m,   
\end{align*} 
so 
$(p_1q_1, \dots, p_nq_n)\in X^{(3)}(F)$ by Lemma \ref{lem.x3F} (3), hence the map 
$$\Psi:\cV(R)\to X^{(3)}(F); (p, q)\mapsto p*q$$ is a morphism of varieties. 

Since $F$ is linearly independent, $X^{(1)}(F)=\emptyset$ by Lemma \ref{lem.liX1} (1), so, for every $(\l_1, \dots, \l_n)\in X^{(3)}(F)$, there exist $p, q\in \PP^{n-1}$ such that 
$f:=\sum_{m=1}^n\l_if_i=g_pg_q$
in $k[u_1, \dots, u_n]$ by Lemma \ref{lem.x3F} (3). Since 
$$S^!=T(V^*)/(R^{\perp})=k[u_1, \dots, u_n]/(f_1, \dots, f_n),$$ 
$g_pg_q\in R^{\perp}$, viewing $g_pg_q$ as an element in $V^*\otimes V^*$, so $(p, q)\in \cV(R)$, hence $\Psi:\cV(R)\to X^{(3)}(F)$ is surjective.    

Since $k[u_1, \dots, u_n]$ is a UFD, a factorization of $f$ is unique up to scalar and the order, so $\Psi:\cV(R)\to X^{(3)}(F)$ is (at most) a double cover.  
\end{proof}

\begin{theorem} \label{thm.0G1} 
Every $S\in \sC_{n, 0}$
satisfies (G1) with $\cP(S)=(E, \s)$ such that $\s^2=\id$.  Moreover, if $S=S^F$ for some normalized $F$,   
then $\Phi:E\to X^{(3)}(F): p\mapsto p*\s(p)$ is (at most) a double cover.  
\end{theorem} 

\begin{proof}
We may assume that $S=S^F$ for some normalized $F$.  Define maps 
\begin{align*}
& \pi_1:
\cV(R)\to \PP^{n-1}; \; (p, q)\mapsto p, \\ 
& \pi_2:
\cV(R)\to \PP^{n-1}; \; (p, q)\mapsto q. 
\end{align*}
Since $R\subset {\rm Sym}^2(V)$ by Proposition \ref{prop.symd}, $(p, q)\in \cV(R)$ if and only if $(q, p)\in \cV(R)$, so  $\pi_1(\cV(R))=\pi_2(\cV(R))=:E$.  
Suppose that $\pi_1:\cV(R)\to \PP^{n-1}$ is not injective.  Then there exist $(p, q), (p, q')\in \cV(R)$ such that $q\neq q'$.   By the uniqueness of the factorization in $k[u_1, \dots, u_n]$, $f_1':=g_pg_q, f_2':=g_pg_{q'}\in k[u_1, \dots, u_n]_2$ are linearly independent.  Since $f_1', f_2'\in \sum_{m=1}^nkf_m$ by the proof of  Theorem \ref{thm.rXF}, we may choose $f_3', \dots, f_n'\in k[u_1, \dots, u_n]_2$ such that $\sum_{m=1}^nkf_m'=\sum_{m=1}^nkf_m$.  Since $f_1', \dots, f_n'\in k[u_1, \dots, u_n]$ is not a regular sequence, 
$$H_{k[u_1, \dots, u_n]/(f_1, \dots, f_n)}(t)=H_{k[u_1, \dots, u_n]/(f_1', \dots, f_n')}(t)\neq (1-t^2)^n/(1-t)^n,$$
which is a contradiction by Lemma \ref{lem.rHn}, so $\pi_1$ is injective.  
By symmetry, $\pi_2$ is also injective, so 
$S$ satisfies (G1) with $\cP(S)=(E, \s)$ where $\s:=\pi_2\pi_1^{-1}\in \Aut E$.  Since $(p, \s(p))\in \cV(R)$ implies $(\s(p), p)\in \cV(R)$,  we have $\s^2=\id$.   

Since the map $E\to \cV(R); \; p\mapsto (p, \s(p))$ is an isomorphism, the map $\Phi:E\to X^{(3)}(F): p\mapsto p*\s(p)$ is (at most) a double cover by Theorem \ref{thm.rXF}.  
\end{proof}

%\begin{remark} 
%Vancliff-Van Rompay-Willaert claimed in their paper \cite{VVW} that every $S\in \sC_{n, 0}$ satisfies (G1), however, they did not discuss injectivities of $\pi_1, \pi_2$, which we believe non-trivial.   
%\end{remark}  

\begin{example} \label{exm.skewpa}
As a simplest example, if $F=(F_1, \dots, F_n)$ where 
 $$(F_m)_{ij} =  \begin{cases} 2 & i=j=m, \\ 0 & \text{otherwise}, \end{cases} \quad 1 \leq m \leq n, $$
then
$$S := C(F) = k\<x_1, \dots, x_n\>/(x_ix_j+x_jx_i)_{1\leq i<j\leq n}\in \sC_{n, 0}.$$  
It is well-known that $\cP(S)=(E, \s)$ where 
\begin{align*}
& E =\bigcap _{1\leq i<j< k\leq n}\cV(x_ix_jx_k)=\bigcup _{1\leq i<j\leq n}\ell_{ij}\subset \PP^{n-1}, \\
& \ell_{ij}=\{(0, \dots, 0, a_i, 0, \dots, 0, a_j, 0, \dots, 0)\in \PP^{n-1}\mid (a_i, a_j)\in \PP^1\},
\end{align*}
and $\s\in \Aut E$ is given by
$$\s|_{\ell_{ij}}(0, \dots, 0, a_i, 0, \dots, 0, a_j, 0, \dots, 0)=(0, \dots, 0, a_i, 0, \dots, 0, -a_j, 0, \dots, 0).$$
On the other hand, 
$$ F_1y_1+\cdots+F_ny_n=\begin{pmatrix} 2y_1 & 0 & \cdots & 0 \\ 0 & 2y_2 & \cdots & 0 \\ \vdots & \vdots & \ddots & \vdots \\ 0 & 0 & \cdots  & 2y_n \end{pmatrix},$$ so 
$$X^{(s)}(F) =\bigcap_{1\leq i_1<\dots <i_s\leq n}\cV(y_{i_1}\cdots y_{i_s})
\subset \PP^{n-1}.$$
In particular,   
\begin{align*}
 X^{(n)}(F) &=\cV(y_1\cdots y_n)\subset \PP^{n-1}, \\
 X^{(3)}(F) &=\bigcap_{1\leq i<j<k\leq n}\cV(y_iy_jy_k)=\bigcup _{1\leq i<j\leq n}\ell_{ij}
\subset \PP^{n-1},  \\
X^{(2)}(F)& =\bigcap_{1\leq i<j\leq n}\cV(y_iy_j)=\bigcup_{1 \leq i \leq n}\{(0, \dots, 0, a_i, 0, \dots, 0)\in \PP^{n-1}\mid a_i=1\} \\
& ={\rm Sing}(X^{(3)}(F))\; \textnormal {(the set of singular points of $X^{(3)}(F)$}), 
\end{align*}   
and $\Phi:E\to X^{(3)}(F)$ is given by 
$$\Phi|_{\ell_{ij}}(0, \dots, 0, a_i, 0, \dots, 0, a_j, 0, \dots, 0)=(0, \dots, 0, a_i^2, 0, \dots, 0, -a_j^2, 0, \dots, 0).$$
\end{example}

\begin{example} \label{exm.phi3dim}
We compute the map $\Phi:E\to X^{(3)}(F)$ for every $3$-dimensional Clifford quantum polynomial algebras $S=C(F) \in \sC_{3,0}$ (see Example \ref{exm.3-dimgcliff}) in Table \ref{3cliqpaphi} of Section 5.
\end{example}

\subsection{Clifford quadratic complete intersections} 

In this subsection, we extend some geometric properties proved in the previous subsection to Clifford quadratic complete intersections.

\begin{definition} 
Let $S=S^F\in \sC_{n, 0}$ for some normalized $F$.  For 
$f\in Z(S)_2$, 
we define 
$$\widetilde K_f:=\{(\l_1, \dots, \l_n)\in \AA^n\mid (\l_1x_1+\cdots +\l_nx_n)^2=f \in S\},  
$$
and $K_f:=\widetilde K_f/\sim$ where $g\sim g'$ if and only if $g'=\pm g$ for $g, g'\in S_1$.  
\end{definition} 

The Clifford deformation of a quadratic algebra was introduced by He and Ye \cite{HY}.  

\begin{definition} 
Let $S=T(V)/(R)$ be a quadratic algebra where $R\subset V\otimes V$. A linear map $\theta : R \to k$ is called a {\it Clifford map} if 
$$(\theta \otimes 1- 1 \otimes \theta)(V \otimes R \cap R \otimes V) = 0.$$ 
We define the {\it Clifford deformation} of $S$ associated to $\theta$ by 
$$
S(\theta) : = T(V) / (f - \theta(f)\mid f \in R).
$$
\end{definition}

Let $S=T(V)/(R)$ be a quadratic algebra where $R\subset V\otimes V$.  Note that the $\ZZ_2$-graded structure on $T(V)=(\oplus _{i=0}^{\infty}V^{\otimes 2i})\oplus (\oplus _{i=0}^{\infty}V^{\otimes 2i+1})$ induces a $\ZZ_2$-graded structure on $S(\theta)$.  
The set of Clifford maps $\theta$ of $S^!=T(V^*)/(R^{\perp})$ is in one-to-one correspondence with the set of central elements $f \in Z(S)_2$ by \cite[Lemma 2.8]{HY}.  We denote by $\theta_f$ the Clifford map of $S^!$ corresponding to $f\in Z(S)_2$.  

We say that $A=S/(f)$ is a noncommutative quadric hypersurface if $S$ is a $d$-dimensional quantum polynomial algebra and $f\in Z(S)_2$ is a regular element.  In particular, we call $A$ a noncommutative conic (resp. quadric) if $d=3$ (resp. $d=4$).  If 
$A=S/(f)$ is a noncommutative quadric hypersurface, then there exists a unique central element $f^!\in Z(A^!)_2$ such that $S^!=A^!/(f^!)$ where $A^!$ is the quadratic dual of $A$ (see Lemma \ref{lem.ffs}).  The finite dimensional algebra $C(A):=A^![(f^!)^{-1}]_0$ plays an important role to study $A$ by the following theorem:    

\begin{theorem} [{\cite[Proposition 5.2 (1)]{SmV}}]
\label{thm.SV}
If $A$ is a noncommutative quadric hypersurface, then 
the stable category $\uCM^{\ZZ}(A)$ of the category of maximal Cohen-Macaulay graded right $A$-modules and the bounded derived category $\sD^b(\operatorname{mod} C(A))$ of the category of finitely generated right $C(A)$-modules are equivalent. 
\end{theorem}

\begin{theorem} \label{thm.ctf} 
Let $S=S^F\in \sC_{n, 0}$ for some normalized $F$, $f\in Z(S)_2$ and $A=S/(f)\in \sC_{n, 1}$.   
\begin{enumerate}
\item{} $\widetilde K_f\cong \Spec S^!(\theta_f)$.  
\item{} $K_f\cong \Spec C(A)\cong \Proj {A^!}$.  
\end{enumerate}
\end{theorem} 

\begin{proof} By Corollary \ref{cor.center}, we may write $f=\sum_{m=1}^na_mx_m^2\in Z(S)_2$.  

(1)  If $g=\sum_{i=1}^n\l_ix_i\in S_1$, then 
\begin{align*}
g^2 & =\sum_{1\leq i, j\leq n}\l_i\l_jx_ix_j=\sum_{m=1}^n\l_m^2x_m^2+\sum_{1\leq i<j\leq n}\l_i\l_j(x_ix_j+x_jx_i) \\
& =\sum_{m=1}^n\l_m^2x_m^2+\sum _{1\leq i<j\leq n}\l_i\l_j\sum_{m=1}^n(F_m)_{ij}x_m^2 \\
&=\sum_{m=1}^n\left (\frac{1}{2}\sum _{1\leq i, j\leq n}\l_i\l_j(F_m)_{ij}\right) x_m^2, 
\end{align*}
so $g^2=f$ if and only if 
$$(\l_1, \dots, \l_n)\in \cV\left (\{\sum_{1\leq i, j \leq n} (F_m)_{ij}u_i u_j  - 2a_m\}_{1\leq m\leq n}\right )\subset \AA^n.$$

On the other hand, since 
$$S^!=B(F)=k[u_1, \dots, u_n]/ (\sum_{1\leq i, j\leq n}(F_m)_{ij}u_iu_j
)_{1\leq m\leq n}, $$
and 
\begin{align*}
\theta_f\left(\sum_{1\leq i, j\leq n}(F_m)_{ij}u_iu_j\right)&=\left(\sum_{1\leq i, j\leq n}(F_m)_{ij}u_iu_j\right)\left (\sum _{i=1}^na_ix_i^2\right)\\
&=\sum _{i=1}^na_i(F_m)_{ii} =2a_m,
\end{align*}
we have 
$$S^!(\theta_f) =k[u_1, \dots, u_n]/(\sum_{1\leq i, j\leq n}(F_m)_{ij}u_iu_j-2a_m)_{1\leq m\leq n},$$
so $\widetilde K_f\cong \Spec S^!(\theta_f)$ (as varieties).  

(2) 
The similar proof as in \cite[Proposition 4.3]{HMM}  shows that 
\begin{align*}
&K_f\cong \Spec (S^!(\theta_f)_0)\cong \Spec C(A)\cong \Proj {A^!}. 
\qedhere
\end{align*}
\end{proof}

For a right noetherian connected graded algebra $A$, following \cite{AZ}, we define the {\it noncommutative projective scheme} associated to $A$ by $\Projn A=(\tails A, \cA)$ where $\tors A$ is the full subcategory of $\grmod A$ consisting of finite dimensional modules over $k$, $\tails A:=\grmod A/\tors A$ is the quotient category, and $\cA$ is the image of $A\in \grmod A$ in $\tails A$.  
We will say that $\Projn A$ is {\it smooth} if  $\tails A$ has finite global dimension.

\begin{corollary} 
For $A\in \sC_{n, 1}$, the following are equivalent:
\begin{enumerate}
\item{} $\Projn A$ is smooth. 
\item{} $C(A)\cong k^{2^{n-1}}$.
\item{} $\#(K_f)=2^{n-1}$.  
\end{enumerate} 
\end{corollary}

\begin{proof}  (1) $\Leftrightarrow$ (2): By \cite[Theorem 6.3]{HY}, $\Projn A$ is smooth if and only if $C(A)$ is semisimple.  Since $A^!\in \sB_{n, n-1}$ is commutative by Theorem \ref{thm.sq}, $C(A)$ is commutative.  Since $\dim_k C(A)={2^{n-1}}$ by \cite[Lemma 5.1]{SmV}, $C(A)$ is semisimple if and only if $C(A)\cong k^{2^{n-1}}$.  

(2) $\Leftrightarrow$ (3).  This follows from Theorem \ref{thm.ctf} (2). 
\end{proof}

\begin{remark} We want to emphasize in general that, for a quantum polynomial algebra $S$, $f\in Z(S)_2$ and $A=S/(f)$,  there is no reason to believe that if $S^!$ is commutative, then $A^!$ is also commutative, so Theorem \ref{thm.sq} is needed to prove the above Corollary.   
\end{remark}

The following theorem is essential to study Clifford quadratic complete intersections.

\begin{theorem} \label{thm.cen} 
If $S\in \sC_{n, 0}$, then $Z(S)_2=\{g^2\in S_2\mid g_1\in S_1\}$. 
In particular,  every $A\in \sC_{n, r}$ is of the form $A=S/(g_1^2, \dots, g_r^2)$ where $S\in \sC_{n, 0}$ and $g_1, \dots, g_r\in S_1$.  
\end{theorem} 

\begin{proof} If $S=C(F)$ for some normalized $F$, then $Z(S)_2=\sum_{m=1}^nkx_m^2$ by Corollary \ref{cor.center}.   
By the proof of Theorem \ref{thm.ctf} (1), $g^2\in \sum_{m=1}^nkx_m^2=Z(S)_2$, so  $Z(S)_2\supset \{g^2\in S_2\mid g_1\in S_1\}$.
On the other hand, for every $f\in Z(S)_2=\sum_{m=1}^nkx_m^2$, $K_f\neq \emptyset$ by Theorem \ref{thm.ctf} (2) so that there exists $g\in S_1$ such that $g^2=f$, so  $Z(S)_2\subset \{g^2\in S_2\mid g_1\in S_1\}$.  Since $S\cong C(F)$ for some normalized $F$ by Corollary \ref{cor.nor}, and the properties $Z(S)_2=\{g^2\in S_2\mid g\in S_1\}$ is preserved by isomorphisms of graded algebras, the result follows.  
\end{proof}

Note that the condition (G1) is essential to study and classify noncommutative conics $A\in \sA_{3, 1}$ in \cite{HMM}, so the following result is important.

\begin{theorem} \label{thm.SsA} Every $A\in \sC_{n, r}$ satisfies (G1).  In fact, 
if $S\in \sC_{n, 0}$ 
with $\cP(S)=(E, \s)$, and $A=S/(g_1^2, \dots, g_r^2)\in \sC_{n, r}$ where $g_1, \dots, g_r\in S_1$, 
then $\cP(A)=(E_A, \s_A)$ where  
$$E_A=(E\cap \cV(g_1, \dots, g_r))\cup \s(E\cap \cV(g_1, \dots, g_r)), \; \s_A=\s|_{E_A}.$$
\end{theorem} 

\begin{proof}  We prove by induction on $r$.  The result holds for $r=0$ by Theorem \ref{thm.0G1}.  Suppose that $A=S/(g_1^2, \dots, g_r^2)\in \sC_{n, r}$ satisfies (G1) where $$E_A=(E\cap \cV(g_1, \dots, g_r))\cup \s(E\cap \cV(g_1, \dots, g_r)), \; \s_A=\s|_{E_A}.$$
Since $\s_A^2=(\s|_{E_A})^2=\s^2|_{E_A}=\id$ by Theorem \ref{thm.0G1}, 
$S/(g_1^2, \dots, g_r^2, g_{r+1}^2)=A/(g_{r+1}^2)\in \sC_{n, r+1}$ satisfies (G1) where  
\begin{align*}
E_{A/(g_{r+1}^2)} & = (E_A\cap \cV(g_{r+1}))\cup \s(E_A\cap \cV(g_{r+1})) \\
&  = (E\cap \cV(g_1, \dots, g_r, g_{r+1}))\cup \s(E\cap \cV(g_1, \dots, g_r, g_{r+1})), \\
\s_{A/(g_{i+1}^2)} & =\s_A|_{E_{A/(g_{r+1}^2)}}=\s|_{E_{A/(g_{r+1}^2)}}
\end{align*}
by \cite[Lemma 4.4]{HMM}.  
\end{proof}

\begin{definition} \label{def.chva} Let $S=S^F\in \sC_{n, 0}$ for some normalized $F$ and $A=S/(g_1^2, \dots, g_r^2)\in \sC_{n, r}$ where $g_1, \dots, g_r\in S_1$.  For $g_i^2=\sum_{j=1}^na_{ij}x_j^2\in Z(S)_2$, we write 
$$
\tilde g_i:=\sum_{j=1}^na_{ij}u_j\in k[u_1, \dots, u_n]_1.$$  
The {\it characteristic varieties} of $A$ are defined by 
$$X^{(s)}_A:=X^{(s)}(F)\cap \cV(\tilde g_1, \dots, \tilde g_r).$$ 
\end{definition}

\begin{theorem} \label{thm.XA} 
If $S=T(V)/(R)=S^F\in \sC_{n, 0}$ for some normalized $F$, and $A=S/(g_1^2, \dots, g_r^2)\in \sC_{n, r}$ where $g_1, \dots, g_r\in S_1$, then $\Phi:E_S\to X^{(3)}_S$
restricts to (at most) a double cover 
$E_A\to X^{(3)}_A$.
\end{theorem}

\begin{proof} 
Write $\G_S:=\cV(R)=\{(p, \s_S(p))\mid p\in E_S\}$ and $\G_A=\G_S\cap \cV(g_1^2, \dots, g_r^2)$. Since $A$ satisfies (G1) by Theorem \ref{thm.SsA},  $\G_A=\{(p, \s_A(p))\mid p\in E_A\}$. For $p\in E_S$ and $q=\s_S(p)\in E_S$ so that $(p, q)\in \G_S$, 
$$\tilde g_i(\Phi(p))=\sum_{j=1}^na_{ij}p_jq_j=g_i^2(p, q),$$ 
so 
\begin{align*}
p\in E_A  & \iff (p, q)\in \G_A \\
& \iff (p, q)\in \G_S,
\; g_i^2(p, q)
=0 \; \forall i=1, \dots, r \\& \iff p\in E_S,
\; g_i^2(p, q)
=0 \; \forall i=1, \dots, r \\
& \iff \Phi(p)\in X^{(3)}_S, \; \tilde g_i(\Phi(p))=0\; \forall i=1\dots, r \\
& \iff \Phi(p)\in X^{(3)}_A,
\end{align*}
hence $\Phi:E_S\to X^{(3)}_S$ restricts to (at most) a double cover 
$E_A\to X^{(3)}_A$. 
\end{proof}

\subsection{Noncommutative conics revisited} 

Noncommutative conics $\sA_{3, 1}$ were classified in \cite{HMM}.  We reclassify Clifford noncommutative conics $\sC_{3, 1}$ in terms of their characteristic varieties.

\begin{lemma} \label{lem.XsA} Let $S=S^F\in \sC_{n, 0}$ for some normalized $F$, $A=S/(g_1^2, \dots, g_r^2)\in \sC_{n, r}$ where $g_1, \dots, g_r\in S_1$, and $\Phi:E_S\to X^{(3)}_S; \; p\to p*\s_S(p)$ the double cover.  
\begin{enumerate}
\item{} For $p\in E_A$,  the following are equivalent: 
\begin{enumerate}
\item{} $\Phi^{-1}(\Phi(p))=\{p\}$.  
\item{} $\s_A(p)=p$. 
\item{} $\Phi(p)\in X^{(2)}_A$.
\end{enumerate}
\item{} The following are equivalent: 
\begin{enumerate}
\item{} $\Phi|_{E_A}:E_A\to X^{(3)}_A$ is strictly 2 to 1. 
\item{} $\s_A$ acts on $E_A$ freely.  
\item{} $X^{(2)}_A=\emptyset$. 
\end{enumerate}
\item{}  $\Phi$
induces isomorphisms 
$E_A/\<\s_A\>\to X^{(3)}_A$ and  $E_A^{\s_A}
% :=\{p\in E_A\mid \s_A(p)=p\}
\to X^{(2)}_A$ where $E_A^{\s_A}:=\{p\in E_A\mid \s_A(p)=p\}$.
\end{enumerate}
\end{lemma} 

\begin{proof}  (1)
(a) $\Leftrightarrow $ (b):  For $p\in E_A$, $\Phi^{-1}(\Phi(p))=\{p, \s_A(p)\}$ by Theorem \ref{thm.0G1}, hence the result. 

(b) $\Leftrightarrow $ (c):  For $p\in E_S$, $p=\s_S(p)$ if and only if $\Phi(p)\in X^{(2)}(F)$ by Lemma \ref{lem.x3F} (2).  For $p\in E_A$, $\Phi(p)\in X^{(3)}_A$, so  $p=\s_A(p)$ if and only if $p=\s_S(p)$ if and only if $\Phi(p)\in X^{(2)}(F)\cap X^{(3)}_A=X_A^{(2)}$.  

(2) and (3) follow directly from (1).  
\end{proof} 

\begin{remark} Let $S=S^F\in \sC_{n, 0}$ for some normalized $F$, and $A=S/(g_1^2, \dots, g_r^2)\in \sC_{n, r}$ where $g_1, \dots, g_r\in S_1$.   By Lemma \ref{lem.XFB}, $X^{(s)}_S$
is independent of the choice of $F$ up to projective equivalence.  Moreover, after fixing a normalized $F$, $X^{(3)}_A=\Phi(E_A)$ and $X^{(2)}_A=\Phi(E_A^{\s_A})$ are  independent of the choice of $g_1, \dots, g_r$ by Theorem \ref{thm.XA} and Lemma \ref{lem.XsA}.  
\end{remark}

\begin{remark} \label{rem.EC} 
If $S\in \sC_{3, 0}$, then $\Phi^{-1}(\Phi(p))=\{p\}$ if and only if $p\in {\rm Sing}(E_S)$ if and only if $\Phi(p)\in X^{(2)}(F)={\rm Sing}(X^{(3)}(F))$ (see Example \ref{exm.phi3dim}).  Note that  $\s_S$ acts on $E_S$ freely if and only if $S$ is a 3-dimensional Sklyanin algebra (Type EC).  
\end{remark}

We say that geometric pairs $(E, \s), (E', \s')$ where $E, E'\subset \PP^{n-1}$ are isomorphic, denoted by $(E, \s)\cong (E', \s')$, if there exists $\t\in \Aut \PP^{n-1}$ which induces an isomorphism $\t:E\to E'$ such that $\t\s=\s'\t$.

\begin{theorem} \label{thm.gepa} 
Let $S=S^F, S'=S^{F'}\in \sC_{n, 0}$ for some normalized $F, F'$, and $A=S/(g_1^2, \dots, g_r^2), A'=S'/({g_1'}^2, \dots, {g'_r}^2)\in \sC_{n, r}$ where $g_1, \dots, g_r\in S_1, g_1', \dots, g_r'\in S'_1$.  For the following three conditions 
\begin{enumerate}
\item{} $A\cong A'$. 
\item{} $\cP(A)\cong \cP(A')$.
\item{} $X^{(3)}_A\cong_p X^{(3)}_{A'}$ which restricts to $X^{(2)}_A\cong_p X^{(2)}_{A'}$. 
\end{enumerate}
we have $(1) \Rightarrow (2) \Rightarrow (3)$.  If $n=3$ and $r=1,2,3$, then $(1), (2), (3)$ are all equivalent.
\end{theorem}

\begin{proof} 

(1) $\Rightarrow $ (2):  This is true for general quadratic algebras $A, A'$ satisfying (G1)  by \cite[Lemma 2.5]{MU}. 

(2) $\Rightarrow$ (3): If $\cP(A)\cong \cP(A')$, then there exists an isomorphism $\t:E_A\to E_{A'}$ such that $\t\s_A=\s_{A'}\t$.  For $p\in E_A$, $p=\s_A(p)$ if and only if $\t(p)=\t\s_A(p)=\s_{A'}(\t(p))$, so $\t$ induces isomorphisms $X_A^{(3)}\cong E_A/\<\s_A\>\to E_{A'}/\<\s_{A'}\>\cong X_A^{(3)}$ and $X_A^{(2)}\cong E_A^{\s_A}\to E_{A'}^{\s_{A'}}\cong X_A^{(2)}$ by Lemma \ref{lem.XsA} (3).
 
Assume now that $n=3$ and $r = 1,2,3$.

$(3) \Rightarrow (1)$: If $X^{(3)}_A\cong_p X^{(3)}_{A'}$ which restricts to $X^{(2)}_A\cong_p X^{(2)}_{A'}$, then 
clearly, $(\#(X^{(3)}_A), \#(X^{(2)}_A))=(\#(X^{(3)}_{A'}), \#(X^{(2)}_{A'}))$.  By direct calculations, 
there are (at least) 6 distinct pairs of numbers appearing as $(\#(X^{(3)}_A), \#(X^{(2)}_A))$ (see Table \ref{tab.pa} of Section 5).  On the other hand,  $\#(\sC_{3, 1})=\#(\sB_{3, 2})=6$ by Theorem \ref{thm.sq} and Remark \ref{rem.Sesy}, hence the result follows. 
\end{proof}

For $A\in \sC_{3,1}$, we can check smoothness of $\Projn A$ by its characteristic varieties.

\begin{corollary} 
For $A\in \sC_{3, 1}$, 
the following are equivalent:
\begin{enumerate}
\item{} $\Projn A$ is smooth. 
\item{} $C(A)\cong k^4$.
\item{} $\#(E_A)=6$.   
\item{} $\#(X^{(3)}_A)=3$ and $X^{(2)}_A=\emptyset$.  
\end{enumerate} 
\end{corollary} 

\begin{proof} (1) $\Leftrightarrow$ (2) $\Leftrightarrow$ (3):  This was shown in \cite [Theorem 5.15]{HMM}.  

(3) $\Leftrightarrow$ (4):  
If $\#(X^{(3)}_A)=3$ and $X^{(2)}_A=\emptyset$, then 
$\Phi:E_A\to X^{(3)}_A$ is strictly 2 to 1 by Lemma \ref{lem.XsA} (2), so $\#(E_A)=6$.  Conversely, suppose that $\#(E_A)=6$.  Since $\#(E_S\cap \cV(g))\leq 3$ by Bezout's theorem (if it is finite), so $E_A$ is the disjoint union of $E_S\cap \cV(g)$ and $\s(E_S\cap \cV(g))$ by Theorem \ref{thm.SsA}.  
It follows that $\s_A$ acts on $E_A$ freely, so $\Phi:E_A\to X^{(3)}_A$ is strictly 2 to 1, hence  $\#(X^{(3)}_A)=3$ and $X^{(2)}_A=\emptyset$ by Lemma \ref{lem.XsA} (2).
\end{proof} 

\newpage

\section{Appendix: Tables}

Recall that, for $a, b, c\in k$, we define the graded algebra
$$
S^{(a,b,c)}:=k\langle x,y,z \rangle/(yz + zy + ax^2, zx + xz + b y^2, xy + yx + cz^2).
$$

\begin{table}[h]
\caption{Geometric pairs of $S^{(a, b, c)}\in \sC_{3, 0}$.} \label{3dimcliqpa}
\begin{tabular}{|c|c|c|}  \hline
Type & $(a,b,c)$ & {\centering  Geometric pair $(E,\sigma)$} \\ \hline  \hline
S & $(0,0,0)$ & \rule{0pt}{31pt}
\parbox{9.2cm}{$E=\mathcal{V}(xyz)$ is a triangle  with singular points $(1,0,0)$, 
$(0,1,0)$,$(0,0,1) \in \mathbb{P}^2$, 
and 
$\left\{
	\begin{array}{ll}
	\sigma(0, b, c)=(0,  b,  -c) \\
	\sigma(a, 0, c)=(-a, 0,  c) \\
	\sigma(a, b, 0)=(a, -b,  0)
	\end{array}
	\right.$ }\\[22pt] \hline 
S' & $(1,0,0)$ &  \rule{0pt}{31pt}
\parbox{9.2cm}{$E=\mathcal{V}(x^3-2 xyz)$ is the union of a smooth conic and 
a line with singular points $(0,1,0)$, $(0,0,1)\in \mathbb{P}^2$, and 
$\left\{
	\begin{array}{ll}
	\sigma(0,b,c)=(0,b, -c)\\
	\sigma(a,b,c)=(a,-b,-c)\\
	\end{array}
	\right. $ }\\[22pt] \hline
NC & $(1,1,0)$ & \rule{0pt}{31pt}
\parbox{9.2cm}{$E=\mathcal{V}(x^3 + y^3 - 2xyz)$ is a nodal cubic curve  with  a singular point 
$(0,0,1) \in \mathbb{P}^2$, and 
$\left
\{\begin{array}{l}
\sigma(a, b, c) = (a,-b,c -\frac{a^2}{b})  \\
\sigma(0,0,1) = (0,0,1).
\end{array}
\right.$ }\\[22pt] \hline
EC & $(\lambda,\lambda,\lambda)$ & \rule{0pt}{67pt}
\parbox{9.2cm}{$E = \mathcal{V}(\lambda(x^3 + y^3 + z^3) - (\lambda^3+2)xyz)$ is an elliptic 
curve, 
and $\sigma$ is the translation by the $2$-torsion point 
$(1,1,\lambda) \in E$  (with the zero element $(1,-1,0) \in E$),  explicitly,
$\s(a, b, c)=\begin{cases} (\l b^2-ac, \l a^2-bc, c^2-\l ^2ab) \\
\textnormal { if } (a, b, c)\in E\setminus E_1; \\
(\l c^2-ab, b^2-\l^2 ac, \l a^2-bc) \\ 
\textnormal { if } (a, b, c)\in E\setminus E_2,
\end{cases}$ 
where $E_1:=\{(1, \e, \l \e^2)\mid \e^3=1\}$,  
$E_2:=\{(1, \l \e^2, \e)\mid \e^3=1\} \subset E$. }\\[12.5ex] \hline
\end{tabular}
\end{table}

% \begin{table}[h]
% \caption{List of $\mathscr{A}_{3, 1}$.} \label{tab-a31}
% \begin{tabular}{|ccc|} 
% \hline
% $k[x, y, z]/(x^2)$, & $k[x, y, z]/(x^2+y^2)$, & $k[x, y, z]/(x^2+y^2+z^2)$, \\ \hline
% $S^{(0, 0, 0)}/(x^2)$, & $S^{(0, 0, 0)}/(x^2+y^2)$, & $S^{(0, 0, 0)}/(x^2+y^2+z^2)$, \\ \hline
% $S^{(1, 0, 0)}/(y^2)$, & $S^{(1, 0, 0)}/(x^2+y^2 + z^2)$, &  $S^{(1, 1, 0)}/(3x^2+3y^2+4z^2)$. \\ \hline
% \end{tabular}
% \begin{tablenotes}
% \item where $S^{(a,b,c)} = k \langle x,y,z \rangle /(yz + zy + a x^2, zx + xz + b y^2, xy + yx + c z^2)$
% \end{tablenotes}
% \end{table}

\begin{table}[h]
\caption{$\Phi:E\to X^{(3)}(F)$ for $S^{(a, b, c)}\in \sC_{3, 0}$.} \label{3cliqpaphi}
\begin{tabular}{|c|c|c|}  \hline
Type  & \rule{0pt}{11.5pt} $\Phi:E\to X^{(3)}(F)$ \\ \hline  \hline
S  & \rule{0pt}{10pt}
\parbox{11.2cm}{
$S^{(0,0,0)}$ is a special case of Example \ref{exm.skewpa}.
}\\%[22pt] 
\hline 
S'  &  \rule{0pt}{51pt}
\parbox{11.2cm}{
$S^{(1,0,0)} = C(F)$ where \\
$
F = \left(\begin{pmatrix} 2 & 0 & 0 \\ 0 & 0 & -1 \\  0 & -1 & 0\end{pmatrix}, \begin{pmatrix} 0 & 0 & 0\\ 0 & 2 & 0 \\ 0 & 0 & 0 \end{pmatrix}, \begin{pmatrix} 0 & 0 & 0 \\ 0 & 0 & 0 \\  0 & 0 & 2 \end{pmatrix}  \right)
$. \\
 $X^3(F) = \mathcal{V}(y_1^3 - 4 y_1y_2y_3)$,\\
$\Phi:  E \to X^3(F); \ 
\left\{
	\begin{array}{ll}
	(0, b, c)\mapsto (0,  b^2,  -c^2) \\
	(a, b, c)\mapsto(a^2, -b^2,  -c^2).
	\end{array}
	\right.$
}
\\[41pt] \hline
NC & \rule{0pt}{51.5pt}
\parbox{11.2cm}{
$S^{(1,1,0)} = C(F)$ where \\
$
F = \left(\begin{pmatrix} 2 & 0 & 0 \\ 0 & 0 & -1 \\  0 & -1 & 0\end{pmatrix}, \begin{pmatrix} 0 & 0 & -1\\ 0 & 2 & 0 \\ -1 & 0 & 0 \end{pmatrix}, \begin{pmatrix} 0 & 0 & 0 \\ 0 & 0 & 0 \\  0 & 0 & 2 \end{pmatrix} \right).
$ \\ 
$X^3(F) = \mathcal{V}(y_1^3 - 4 y_1y_2y_3)$, \\ 
$\Phi:  E \to X^3(F); \ 
\left\{
	\begin{array}{ll}
	(a, b, c)\mapsto (a^2,  -b^2,  c^2 - \frac{a^2c}{b}) \\
	(0, 0, 1)\mapsto(0, 0, 1).
	\end{array}
	\right.$
}
\\[36pt] \hline
EC  & \rule{0pt}{73pt}
\parbox{11cm}{
$S^{(\lambda,\lambda,\lambda)} = C(F)$ where \\
$
F = \left(\begin{pmatrix} 2 & 0 & 0 \\ 0 & 0 & -\lambda \\  0 & -\lambda & 0\end{pmatrix}, \begin{pmatrix} 0 & 0 & -\lambda \\ 0 & 2 & 0 \\ -\lambda & 0 & 0 \end{pmatrix}, \begin{pmatrix} 0 & -\lambda & 0 \\ -\lambda & 0 & 0 \\  0 & 0 & 2 \end{pmatrix}  \right),
$ \\
$\lambda \in k$ such that $\lambda^3 \neq 0,1,-8$. \\
$X^{(3)}(F)= \mathcal{V}(\lambda^2 (y_1^3 + y_2^3 + y_3^3) - (4-\lambda^3)y_1 y_2 y_3),$ \\ 
$\Phi:  E \to X^3(F); \ (a, b, c)\mapsto$ \\
$ 
\begin{cases} (a(\l b^2-ac), b(\l a^2-bc), c(c^2-\l ^2ab)) & \textnormal { if } (a, b, c)\in E\setminus E_1 \\
(a(\l c^2-ab), b(b^2-\l^2 ac), c(\l a^2-bc)) & \textnormal { if } (a, b, c)\in E\setminus E_2 \end{cases}
$
where
$E_1:=\{(1, \e, \l \e^2)\mid \e^3=1\}, E_2:=\{(1, \l \e^2, \e)\mid \e^3=1\} \subset E$.
}
\\[13.5ex] \hline
\end{tabular}
\end{table}

\begin{landscape}

\begin{table}[ht]\small
%\small
%\centering 
\begin{threeparttable}
\caption{$\sC_{3,r}$ for $r = 1,2,3$.} \label{tab.pa}
\begin{tabular}{|c|c|c|c|c|c|} 
\hline
$A \in \sC_{3,1}$  
& $\mathcal{P}(A)$ 
& $(\#(X_A^{(3)}),\#(X_A^{(2)}))$
&  $E_{A^!} $ 
\\ \hline \hline
$S^{(0,0,0)}/(x^2)$ 
& \begin{tabular}{c}$E_A$ is a line, \\ $\sigma_A$ fixes exactly $2$ points
\end{tabular} 
& $(\infty,2)$
& \begin{tabular}{c}$E_{A^!}$ consists of $1$ point, \\ with Segre sym.: $[1, 1 ; \ ; 1]$ \end{tabular} 
\\ \hline
$S^{(1,0,0)}/(y^2)$ 
& \begin{tabular}{c}$E_A$ consists of $1$ point,\\ $\sigma_A$  maps  the point to itself  \end{tabular} 
& $(1,1)$
& \begin{tabular}{c}$E_{A^!}$ consists of $1$ point, \\ with Segre sym.: $[(2, 1)]$ \end{tabular}  
\\ \hline
$S^{(1,1,0)}/({3}x^2 + {3}y^2 + 4z^2)$
&  \begin{tabular}{c}$E_A$ consists of $2$ points, \\
$\sigma_A$ switches these $2$ points
\end{tabular}  
& $(1,0)$
& \begin{tabular}{c}$E_{A^!}$ consists of $2$ points, \\ with Segre sym.: $[3]$ \end{tabular} 
\\ \hline
$S^{(0,0,0)}/(x^2 + y^2)$ 
&  \begin{tabular}{c}$E_A$ consists of $3$ points in general position, \\ $\sigma_A$ permutes these $3$ points as: 
{\begin{tikzpicture}[x=10, y=10, baseline=(current bounding box)]
\draw[white] (0, -1) -- (0, 1) ;
\fill (2, 2/3) circle (1.2pt) ; 
\fill (2, -2/3) circle (1.2pt) ;
\fill (0, 0.3)  circle (1.2pt) ;
\draw[->] (0,0.3)  arc (90:400:1.5mm) ;
\draw[<->]
      (2.3,2/3)
      to[out=-10, in=30] (2.3,-2/3) ;
\end{tikzpicture}}\end{tabular} 
& $(2,1)$
& \begin{tabular}{c}$E_{A^!}$ consists of $2$ points,\\ with Segre sym.: $[(1, 1), 1]$\end{tabular}
\\ \hline
$S^{(1,0,0)}/(x^2 + y^2 + z^2)$ 
&  \begin{tabular}{c}$E_A$ consists of $4$ points with $3$ points collinear,\\ $\sigma_A$ permutes these $4$  points as:
 \begin{tikzpicture}[x=10, y=10, baseline=(current bounding box)]
\draw[white] (0, -0.8) -- (0, 2.5) ;
\draw[dashed] (-1.3, 0) -- (4, 0) ; 
\fill (-1, 0) circle (1.2pt) ;
\fill (0, 2) circle (1.2pt) ;
\fill (3/2, 0) circle (1.2pt) ;
\fill (3, 0) circle (1.2pt) ;
\draw[<->]
      (-1+0.08, -0.15-0.12)
      to[out=-30, in=-150] (3/2-0.1, -0.25);
      \draw[<->]
      (0.15+0.13, 2)
      to[out=-10, in=130] (3-0.03, 0.13+0.15);
\end{tikzpicture} \end{tabular} 
& $(2,0)$
& \begin{tabular}{c}$E_{A^!}$ consists of $3$ points,  \\ with Segre sym.: $[2, 1]$ \end{tabular}
\\ \hline
$S^{(0,0,0)}/(x^2 + y^2 + z^2)$ 
&  \begin{tabular}{c}$E_A$ consists of $6$ points in a  quadrilateral \\ configuration,  $\sigma_A$  permutes these $6$ points as: \\
\begin{tikzpicture}[x=11, y=11, baseline=(current bounding box)]
\draw[dashed] (-3.5, 0) -- (0.8, 0) ; 
\draw[dashed] (-2.2, -3.3) -- (1/2, 3/4) ; 
\draw[dashed] (-1.4, 0.6) -- (-2.08, -3.5) ; 
\draw[dashed] (-3.5, 0.375) -- (-0.2, -2.1) ; 
\fill (0, 0) circle (1.2pt) ; 
\fill (-3, 0) circle (1.2pt) ;
\fill (-1.5, 0) circle (1.2pt) ;
\fill (-1, -3/2) circle (1.2pt) ;
\fill (-2, -3) circle (1.2pt) ;
\fill (-5/3, -1) circle (1.2pt) ;
\draw[<->]
      (-0.15, -0.11)
      to[out=-150, in=30] (-5/3+0.16, -0.92);
\draw[<->]
      (-3.08, -0.18)
      to[out=-110, in=150] (-2.17, -2.9);
\draw[<->]
      (-1.4, -0.15)
      to[out=-60, in=100] (-1, -1.32);
\end{tikzpicture}\end{tabular} 
& $(3,0)$
& \begin{tabular}{c}$E_{A^!}$ consists of $4$ points, \\ with Segre sym.: $[1, 1, 1]$\end{tabular} 
\\ \hline \hline
$A \in \sC_{3,2}$ 
& $\mathcal{P}(A)$ 
& $(\#(X_A^{(3)}),\#(X_A^{(2)}))$
&  $E_{A^!}$ 
\\ \hline \hline
$S^{(0,0,0)}/(x^2, y^2)$ 
& \begin{tabular}{c}$E_A$ consists of $1$ point, \\
$\sigma_A$  maps the point to itself   \end{tabular}  
& $(1,1)$
& \begin{tabular}{c}$E_{A^!}$ is a (double) line \end{tabular} 
\\ \hline
$S^{(0,0,0)}/(x^2 + y^2, z^2)$ 
 & \begin{tabular}{c}$E_A$ consists of $2$ points, \\
$\sigma_A$ switches these $2$ points
 \end{tabular}  
 & $(1,0)$
& \begin{tabular}{c}$E_{A^!}$  is the union of $2$ lines in  \\ general position \end{tabular}  
 \\ \hline
$S^{(0,0,0)}/(x^2 + y^2,  x^2 + z^2 )$ 
 & $E_A = \emptyset$  
 & $(0,0)$
 & \begin{tabular}{c}$E_{A^!}$  is a smooth conic \end{tabular} 
 \\ \hline \hline
$A \in \sC_{3,3}$ 
 & $\mathcal{P}(A)$ 
 & $(\#(X_A^{(3)}),\#(X_A^{(2)}))$
 & $E_{A^!}$
 \\ \hline
$S^{(0,0,0)}/(x^2,y^2,z^2)$ & $E_A = \emptyset$
& $(0,0)$
&  $\mathbb{P}^2$  
\\ \hline
\end{tabular}
\begin{tablenotes}
\item For the notation of Segre symbol, we follow \cite{W}.
\end{tablenotes}
\end{threeparttable}
\end{table}

\end{landscape}

\end{document}